\documentclass[review,hidelinks,onefignum,onetabnum]{siamart220329}
\usepackage{inputenc}
\usepackage{bbm}
\usepackage{booktabs}
\usepackage{amsfonts}
\usepackage{pythonhighlight}
\usepackage{longtable}
\usepackage{amsmath}
\usepackage{adjustbox}
\usepackage{siunitx}
\usepackage{tabularx}
\usepackage{float}
\usepackage{latexsym}
\usepackage{graphicx}
\usepackage{subcaption}
\usepackage{ragged2e}
\usepackage{siunitx}
\usepackage{svg}
\usepackage{pythonhighlight}
\usepackage{array}
\usepackage{amsmath,amssymb,mathabx} 
\usepackage{calc}
\theoremstyle{remark}

\usepackage{algorithm}
\usepackage{algpseudocode}
\algnewcommand{\LineComment}[1]{\State \(\triangleright\) #1}

\nolinenumbers

\newcommand{\mc}[1]{\mathcal{#1}}
 
\renewcommand{\vec}[1]{\boldsymbol{#1}}

\newcommand{\mat}[1]{\boldsymbol{#1}}
\newcommand{\B}[1]{\boldsymbol{#1}}
\newcommand{\R}{\mathbb{R}}
\newcommand{\ten}[1]{\boldsymbol{\mathcal{#1}}}

\newcommand{\tenh}[1]{\widehat{\boldsymbol{\mc{#1}}}}
\newcommand{\mathat}[1]{\widehat{\boldsymbol{#1}}}

\newcommand{\bmat}[1]{\begin{bmatrix}#1\end{bmatrix}}

\renewcommand{\t}{^\top}

\makeatletter
\newcommand{\hathat}[1]{%
\begingroup%
  \let\macc@kerna\z@%
  \let\macc@kernb\z@%
  \let\macc@nucleus\@empty%
  \hat{\raisebox{.37ex}{\vphantom{\ensuremath{#1}}}\smash{\hat{#1}}}%
\endgroup%
}
\makeatother

\title{Parametric kernel low-rank approximations using tensor train decomposition\thanks{Submitted to the editors DATE.
\funding{This work was funded by the National Science Foundation through the awards DMS-1845406, DMS-1821149, and 
DMS-2026830.}}}

\author{Abraham Khan\thanks{Department of Mathematics, North Carolina State University (\email{awkhan3@ncsu.edu}).}
\and Arvind K.\ Saibaba\thanks{Department of Mathematics, North Carolina State University (\email{asaibab@ncsu.edu}).}}

\begin{document}
\maketitle
\begin{abstract}
   Computing low-rank approximations of kernel matrices is an important problem with many applications in scientific computing and data science. We propose methods to efficiently approximate and store low-rank approximations to kernel matrices that depend on certain hyperparameters. The main idea behind our method is to use multivariate Chebyshev function approximation along with the tensor train decomposition of the coefficient tensor. The computations are in two stages: an offline stage, which dominates the computational cost and is parameter-independent, and an online stage, which is inexpensive and instantiated for specific hyperparameters. A variation of this method addresses the case that the kernel matrix is symmetric and positive semi-definite. The resulting algorithms have linear complexity in terms of the sizes of the kernel matrices. We investigate the efficiency and accuracy of our method on parametric kernel matrices induced by various kernels, such as the Mat\'ern kernel, through various numerical experiments on problem setups up to three spatial dimensions and two parameter dimensions. Our methods have speedups up to $200\times$ in the online time compared to other methods with similar complexity and comparable accuracy.
\end{abstract}
\section{Introduction}
A kernel is a function of type $\kappa: \R^d \times \R^d \rightarrow \R$. Given two sets of points $\mathcal{X} = \{\vec{x}_1, \dots, \vec{x}_{N_s}\} \subseteq \mathbb{R}^{d}$ and $\mathcal{Y} = \{\vec{y}_1, \dots, \vec{y}_{N_t} \} \subseteq \mathbb{R}^{d}$, which we refer to as the source and target points respectively, we define a kernel matrix $\mat{K}(\mathcal{X},\mathcal{Y}) \in \mathbb{R}^{N_s \times N_t}$ with entries
$$[\mat{K}(\mathcal{X},\mathcal{Y})]_{i,j} = \kappa(\vec{x}_{i},\vec{y}_j), \qquad 1 \leq i \leq N_s, 1 \leq j \leq N_t.$$
Kernel matrices arise in many instances in applied mathematics; for example, Gaussian Processes, N-body simulations, integral equations, etc. However, kernel matrices are dense (depending on the support of $\kappa$) and, therefore, a central challenge is storing and computing kernel matrices when the number of source and target points $N_s$ and $N_t$ are large.  The na\"ive approach has a complexity $\mc{O}(N_sN_t)$ floating point operations (flops), which is not practically applicable. Hence, it is desirable to develop low-rank approximation methods that have linear complexity in the number of source and target points.

The singular value decay of a kernel matrix $\mat{K}(\mathcal{X},\mathcal{Y})$ is dependent on both the spatial configuration and the interactions between source and target points, as well as the smoothness of the kernel function $\kappa$. To quantify the level of interaction between the source and target points, we introduce the source box $\mathcal{B}_{s} =  \bigtimes_{i=1}^d [\alpha^{s}_i, \beta^{s}_i] $ and the target box $\mathcal{B}_{t} =  \bigtimes_{i=1}^d [\alpha^{t}_i, \beta^{t}_i]$, such that $\mathcal{X} \subseteq \mathcal{B}_{s}$ and $\mathcal{Y} \subseteq \mathcal{B}_{t}$, the boxes $\mathcal{B}_{s}$ and $\mathcal{B}_{t}$ can have varying degrees of intersection.  Identifying and computing low-rank approximations of kernel matrices when the boxes $\mc{B}_{s}$ and $\mc{B}_{t}$ are separated, is an important step in many kernel evaluation algorithms such as the fast multipole method and hierarchical matrices.

This paper tackles a slightly more general problem setting. In many applications, the behavior of a kernel function $\kappa$ depends on several parameters (sometimes denoted $\vec\theta \in \Theta$) that influence its behavior and by extension the behavior of the corresponding kernel matrix. In these applications, often the kernel low-rank approximation must be computed efficiently for several instances of the parameters $\vec\theta$. More concretely, given a (parametric) kernel function $\kappa:\R^d\times \R^d  \times \R^{d_\theta} \rightarrow \R$ we define the corresponding parametric kernel matrix as
\begin{equation}\label{eqn:matparamkernel}[\mat{K}(\mathcal{X},\mathcal{Y}; {\vec{\theta}})]_{i, j} = \kappa(\vec{x}_{i},\vec{y}_j; {\vec{\theta}}), \qquad \vec{x}_i \in \mathcal{X}, \quad \vec{y}_j \in \mathcal{Y}, \quad \vec{\theta} \in \Theta,\end{equation}
for $1 \le i \le N_s$ and $1 \le j \le N_t$. Similar to the source and target boxes, the parameter space $\Theta$ is enclosed in a parameter box $\mathcal{B}_{\vec\theta} =  \bigtimes_{i=1}^{d_{\theta}} [\alpha^{p}_i, \beta^{p}_i]$. Our goal is to construct a parametric low-rank approximation of the form
 \begin{equation}\label{eqn:paramlowrank}
     \mat{K}(\mathcal{X},\mathcal{Y};{\vec{\theta}}) \approx \mat{SH}(\B\theta)\mat{T}^\top,
 \end{equation}
 where the matrices $\mat{S} \in \R^{N_s \times s}$ and $\mat{T} \in \R^{N_t \times t}$ represent the factor matrices corresponding to the source and target matrices respectively. Importantly, these factor matrices are independent of the parameter $\vec\theta$, and the parametric dependence only enters through the core matrix $\mat{H}(\vec\theta) \in \R^{s\times t}$. We assume that the ranks satisfy $\max\{s,t\} \ll \min\{N_s,N_t\}$. The advantage of this representation is clear if much of the computation can be offloaded to an offline stage, and the cost of evaluating $\mat{H}(\vec\theta)$ for a new parameter value in the online stage is much cheaper than evaluating the kernel matrix from scratch. A special but important case is when the kernel is symmetric (and perhaps positive definite) and the source and target points coincide. We refer to this as the \textit{global} case, and the resulting low-rank approximation the \textit{global low-rank approximation}.

There are several situations where such a (parametric) low-rank approximation is beneficial. In Gaussian processes, one has to optimize over the hyperparameters that define the kernel. This requires expensive evaluations of the kernel matrix, and perhaps additionally its derivatives with respect to the hyperparameters, at each step of the optimization routine. This similar situation also arises in hierarchical Bayesian inverse problems, where Gaussian processes are used to define the prior distributions, and there are additional hyperparameters to be determined~\cite{hall2023efficient}. The parametric low-rank representation~\eqref{eqn:paramlowrank} can be used instead of the kernel matrix in an outer loop computation, or can be used as a preconditioner to accelerate an iterative method (e.g., conjugate gradients) for the solution of the linear system involving the kernel matrix.

\paragraph{Our approach}
At the heart of our approach is a multivariate Chebyshev interpolation of the function $f_\kappa: \mathbb{R}^{D} \rightarrow \mathbb{R}$ where $D = 2d+d_\theta$ and is defined as 
\begin{equation}\label{eqn:fkappa} f_\kappa(\vec{\xi})  = \kappa(\vec{x},\vec{y}; {\vec{\theta}}), \qquad  \vec{\xi} = (\vec{x}, \vec{\theta}, \vec{y}).\end{equation}
Note that the specific ordering of the variables is essential for our algorithm. For example, in a typical problem setting involving three spatial dimensions $d=3$ and two parameter dimensions $d_\theta =2$, the resulting dimension $D=8$ makes Chebyshev interpolation challenging. Consequently, a coefficient expansion of $n$ points in each dimension gives a coefficient tensor of size $n^D$, which can be challenging to construct and store, for example, when $n = 27$. To address this challenge, we further compress the coefficient tensor in the tensor train (TT) format without explicitly forming it first. This roughly comprises the offline stage. In the online stage, a small amount of work is needed independent of the number of source and target points, and new kernel evaluations are not required.

 \paragraph{Contributions}
The specific contributions and features of our work are as follows:
\begin{enumerate}
    \item We propose a new algorithm, Parametric Tensor Train Kernel (PTTK), that computes the low-rank approximation of parametric kernel matrices by combining multivariate Chebyshev interpolation with the TT-decomposition, \\ when the source and target points are either weakly or strongly admissible.  Our algorithm involves a two-stage process, comprising an offline (pre-computation) stage and an online stage. A version of this algorithm also applies to non-parametric kernel low-rank approximations.
    \item We perform a detailed analysis of the cost involving computations and storage. The offline cost is linear in the number of source and target points (i.e., w.r.t. $N_s$ and $N_t$). However, during the online stage, when a specific instance of $\vec\theta \in \Theta$ is provided, the runtime becomes independent of $N_s$ and $N_t$. The algorithms are accompanied by an analysis of the error.
    \item We adapt our algorithm to handle a global low-rank approximation where the source and target points coincide (i.e., $\mc{X}=\mc{Y}$). If, additionally, the kernel matrix is symmetric and positive semi-definite, we adapt our approach to ensure these properties are retained by the low-rank approximation with a small additional cost in the online stage that does not affect the overall complexity. 
    \item We demonstrate the efficiency of our methods on various non-parametric and parametric kernels with applications including multivariate interpolation, partial differential and integral equations, and Gaussian processes. For weakly or strongly admissible boxes, the algorithms can achieve moderate to high accuracy $10^{-5}-10^{-9}$, whereas for the global low-rank case the algorithms achieve low to moderate accuracy $10^{-2}-10^{-5}$ depending on the smoothness of the kernels. We compare our implementation against other algorithms also with linear complexity, but our approach outperforms existing methods due to the low online time (up to $200\times$ speedup).
    \item  We propose a new method for initializing the index sets used in the tensor train greedy-cross algorithm that may be of independent interest beyond this paper. 
\end{enumerate}

We also note that the numerical experiments show good performance on problems up to $D= 8$ dimensions (three spatial and two parameter dimensions). Future work involves investigating the performance of the methods in higher dimensions. The code to reproduce the tables is located at \url{https://github.com/awkhan3/ParametricTensorTrainKernel}.

.
\paragraph{Related work} We discuss some related work for low-rank kernel approximations, but since this is a vast area, we limit ourselves to a few references to give some context. 
We first consider the case of obtaining a low-rank approximation of a kernel matrix, that is not necessarily parametric. An optimal low-rank approximation, in an orthogonally invariant norm, can be computed using the truncated SVD. However, the cost of this method is prohibitively expensive for large values of $N_s, N_t$ since it involves forming the kernel matrix first and then compressing it. Approaches such as interpolatory decompositions and randomized SVD~\cite{halko} can alleviate this computational cost, but are still very expensive. It is desirable to develop algorithms that have linear dependence on $N_s$ and $N_t$. In this category, there are certain versions of  Adaptive Cross Approximation (ACA) \cite{bebendorf2000approximation, bebendorf2003adaptive}, black-box fast multipole method \cite{fong2009black}, and Nystr\"om approximation~\cite{Cai2023}. Certain analytic techniques such as multipole expansions~\cite{greengard1987fast}, Taylor expansions, equivalent densities~\cite{ying2004kernel}, and proxy point method~\cite{ye2020analytical}. For the global case, and for kernels that define positive semi-definite matrices there are various techniques such as Nystr\"om's method~\cite{NIPS2000_19de10ad}, random Fourier features~\cite{randomfeatures-1,randomfeatures-2}, and randomly pivoted Cholesky~\cite{RPCholesky}. There is also the hybrid cross-approximation method \cite{borm2005hybrid}, which uses polynomial interpolation to obtain a first approximation and then uses a cross-approximation technique to obtain the final low-rank approximation.

To our knowledge, the parametric case has been addressed only in a handful of papers~\cite{Saibaba2022, paz, DK2020} specifically for kernel matrices. In~\cite{https://doi.org/10.1002/nla.2576, park2024lowrankapproximationparameterdependentmatrices}, the authors treat the more general case of parameter-dependent low-rank approximations. We now discuss the relationship between our work and the above references for kernel matrices. The closest work to the present paper is \cite{Saibaba2022}, which also uses a multivariate Chebyshev expansion followed by a compression of the coefficient tensor using Tucker decomposition. The main differences with the above approach are that (1) we use the TT-decompsition, which can potentially scale to higher dimensions than the Tucker decomposition (2) we use the TT-cross approximation, which avoids forming the tensor explicitly. As in~\cite{Saibaba2022}, our approach also has linear cost in $N_s$ and $N_t$ but has a more favorable computational complexity with respect to the number of Chebyshev points and the dimensions. Our work is also related to~\cite{strossner2022approximation}, which does not deal specifically with kernel approximations but develops algorithms for high-dimensional function approximation.

In \cite{DK2020}, obtaining a low-rank approximation of a parametric kernel matrix that is symmetric and positive semi-definite is done by applying ACA to an affine linear expansion of $\mat{K}(\mathcal{X}, \mathcal{X}; \vec\theta)$ w.r.t. $\vec\theta$, and this affine linear expansion is obtained through the empirical interpolation method. The goal is to use precomputation (offline stage) to obtain a low-rank factorization of the parametric kernel matrix where each factor depends on $\vec\theta$; hence, the online time of the proposed method is dependent on $N_s$ and $N_t$. Our method is not constrained to the global symmetric case and has a more efficient online stage. 

In \cite{paz}, a shift-invariant kernel is expressed in terms of a feature map, as a consequence of Bochner's Theorem, and  a Gauss-Legendre quadrature-based approach is used that gives a low-rank approximation in the form of~\eqref{eqn:paramlowrank}.
As with our proposed algorithm, the online time does not depend on $N_s$ nor $N_t$. However, as in~\cite{Saibaba2022}, the method has an exponential dependence on the number of features and is expensive for even modest dimensions. Furthermore, it also appears to be constrained to the global symmetric case  and requires explicit knowledge of the feature map; our method does not suffer from these issues since it only relies on kernel evaluations.

\section{Background}
\label{ssec:matrix_operations}
\subsection{Matrix Operations}
In this section, we define some matrix operations that will come in handy when dealing with tensors.  Consider two arbitrary matrices $\mat{G} \in \R^{s \times m}$ and $\mat{H} \in \R^{p \times k}$. 
We define the Kronecker product $\mat{G} \otimes \mat{H}$
to be a $\R^{sp \times mk}$ block matrix:

\[
\mat{G} \otimes \mat{H} = 
\begin{bmatrix}
    g_{1,1}\mat{H} & g_{1,2}\mat{H} & \cdots & g_{1,m-1}\mat{H} & g_{1,m}\mat{H} \\
    g_{2,1}\mat{H} & g_{2,2}\mat{H} & \cdots & g_{2,m-1}\mat{H} & g_{2,m}\mat{H} \\
    \vdots & \vdots & \ddots  & \vdots & \vdots \\
    g_{s,1}\mat{H} & g_{s,2}\mat{H} & \cdots & g_{s,m-1}\mat{H} & g_{s,m}\mat{H} \\
\end{bmatrix}.
\]

We now define matrices $\mat{A} \in \R^{s \times q}, \mat{B} \in \R^{p \times q} $. We denote the symbol $\rtimes$ as the Khatri-Rao product, which we define as

\[
\begin{aligned}
\mat{A} \rtimes \mat{B} &= 
\left[
\begin{array} { c | c | c | c}
\vec{a}_1 \otimes \vec{b}_1 & \vec{a}_2 \otimes \vec{b}_2 & \dots & \vec{a}_q \otimes \vec{b}_q
\end{array}
\right] \in \R^{sp \times q}
\end{aligned} 
\]

where $\vec{a}_{i}$ and $\vec{b}_{i}$ are the column vectors of matrices $\mat{A}$ and $\mat{B}$, for $1 \le i \le q$, respectively. We denote the symbol $\ltimes$ as the face-splitting product, also referred to as the  transposed Khatri-Rao product, which we define in terms of the Khatri-Rao product:
$$\mat{A}^{\top} \ltimes \mat{B}^{\top} = (\mat{A} \rtimes \mat{B})^{\top}.$$
Lastly, both the face-splitting product and Khatri-Rao product satisfy some mixed-product properties with the Kronecker product 
\begin{equation}\label{eqn:mixedprod}
\begin{aligned}
(\mathbf{G}^{\top} \otimes \mathbf{H}^{\top})(\mathbf{A} \rtimes \mathbf{B}) = (\mathbf{G}^{\top}\mathbf{A}) \rtimes (\mathbf{H}^{\top}\mathbf{B}), \\
(\mathbf{A}^{\top} \ltimes \mathbf{B}^{\top})(\mathbf{G} \otimes \mathbf{H}) = (\mathbf{A}^{\top}\mathbf{G}) \ltimes (\mathbf{B}^{\top}\mathbf{H}).
\end{aligned}
\end{equation}

\subsection{Tensor Basics}
We define a tensor $\ten{X} \in \mathbb{R}^{n_1 \times n_2 \times \dots \times n_N}$ for $N \in \mathbb{N}$ as a multi-dimensional array. Indexing into a tensor $\ten{X}$ is represented by $x_{i_1, i_2, \dots, i_N}$ or alternatively $[\ten{X}]_{i_1,\dots,i_N}$, where each index $i_j$ ranges from $1$ to $n_j$ for $1 \leq j \leq N$. We define the Chebyshev norm of $\ten{X}$ as
$$\|\ten{X}\|_C = \max_{i_1,\dots,i_N}|x_{i_1, i_2, \dots, i_N}|.$$
Similarly, the Frobenius norm of a tensor $\ten{X}$ is defined as the square root of the sum of squares of all the entries and is given by the formula
\[
\| \ten{X} \|_F = \sqrt{ \sum_{i_1 = 1}^{n_1} \dots \sum_{i_N = 1}^{n_N} |x_{i_1,i_2,\dots,i_N}|^2}.
\]
We now describe the mode-k product that will be used in this paper. For a matrix, $\mat{A} \in \mathbb{R}^{m \times n_k}$, one can define the mode-$k$ product of $\ten{X}$ w.r.t.\  $\mat{A}$ as
$\ten{Y} = \ten{X} \times_{k} \mat{A}$ where the tensor $\ten{Y}$ has entries
$$y_{i_1, \dots, i_{k-1}, j, i_{k+1}, \dots, i_N} = \sum^{n_k}_{i_k=1}x_{i_1,\dots, i_d}a_{j, i_k} , \qquad 1 \le j \le m.$$ Let $(i_1, i_2, \dots i_j) $ be an arbitrary multi-index for $1 \le j \le N$ where $1 \le i_j \le n_j$, and we denote  $ \overline{i_1, i_2, \dots, i_j}$ to be the little endian flattening of the multi-index into a single index defined by the formula
$$\overline{i_1, i_2, \dots, i_j}= i_1 + (i_2 - 1) n_1 + (i_3 - 1)  n_1  n_2 + \ldots + (i_j - 1)  n_1 n_2  \ldots  n_{j-1}.$$
We can now define the tensor unfolding matrix  $\mat{X}^{\{j\}}$ of size $\prod^{j}_{i=1}n_i \times \prod_{i=j+1}^{N} n_i$ whose entries are given by
$$ [\mat{X}^{\{j\}}]_{\overline{i_1, \dots, i_j}, \overline{i_{j+1}, \dots, i_N}} = x_{i_1,i_2,\dots,i_{N-1},i_N}, \quad 1 \le i_j \leq n_j, 1 \le j \le N.$$
 We can also  define the tensor unfolding matrix through MATLAB's reshape command
$$\mat{X}^{\{j\}} = \textrm{reshape}(\ten{X}, \prod_{i=1}^j n_i,  \prod_{i=j+1}^N n_i), \qquad 1 \leq j \leq  N.$$ Lastly, considering an arbitrary matrix $\mat{A} \in \R^{m \times k}$,  we now define the twist tensors $\ten{T}^{(1)}(\mat{A}) \in \R^{1 \times m \times k}$, $\ten{T}^{(2)}(\mat{A}) \in \R^{m \times 1 \times k}$ and $\ten{T}^{(3)}(\mat{A}) \in \R^{m \times k \times 1}$ of $\mat{A}$ as the tensors with entries 
\begin{equation}\label{eqn:twist}[\ten{T}^{(1)}(\mat{A})]_{1, i, j} = a_{i, j}, \quad [\ten{T}^{(2)}(\mat{A})]_{i, 1, j} = a_{i, j}, \textrm{ and } \quad [\ten{T}^{(3)}(\mat{A})]_{i, j, 1} = a_{i, j} ,\end{equation}
for $1 \le i \le m$ and $1 \le j \le k$.
\subsection{Tensor Train format} A tensor $\ten{X}$ is said to be in the Tensor Train (TT) format if $\ten{X}$ can be represented by a chain of third order tensors $\ten{G}_1,\dots,\ten{G}_N$ where $\ten{G}_j \in \mathbb{R}^{r_{j-1} \times n_j \times r_{j}}$ for $1 \leq j \leq N$ (with the convention  $r_0 = r_{N} = 1$) are referred to as the TT-cores and $r_0, \dots, r_N$ as the TT-ranks.
The entries of tensor $\ten{X}$ are  given by the following formula
$$x_{i_1,\dots,i_N} = \sum^{r_1}_{s_1 = 1} \dots \sum^{r_{N-1}}_{s_{N-1}=1} [\ten{G}_1]_{1, i_1, s_1} [\ten{G}_2]_{s_1, i_2, s_2} \cdots [\ten{G}_N]_{s_{N-1}, i_N, 1},$$
and  we also designate $r$ to be $r = \underset{i}{\max}~r_i$.
We denote the  TT-decomposition of $\ten{X}$ as
$\ten{X} = [\ten{G}_1, \ten{G}_2, \dots,\ten{G}_N ].$ Note, the TT-ranks are related to the ranks of the unfoldings of $\ten{X}$; as in,  $r_j = \textrm{rank}(\mat{X}^{\{j\}}).$ It is typically the case that $\ten{X}$ does not have an exact representation in the TT-format, but it can be approximated by the TT-format (we will denote such an approximation as $\tenh{X}$) via the TT-SVD algorithm which is given in~\cite[Algorithm 1]{oseledets2011tensor}.
~Often the TT-ranks of $\tenh{X}$ are inflated if such a TT approximation is obtained through a cross approximation technique, Algorithm~\ref{alg:greedy}. Hence, there is a need for a fast procedure that reduces the inflated TT-ranks of $\tenh{X}$ while maintaining the accuracy of the approximation to $\ten{X}$. We refer to such a procedure as TT rounding. The TT rounding algorithm was introduced in \cite{oseledets2011tensor}, and we describe it in Algorithm \ref{alg:tt_round}. Lastly, the computational cost of Algorithm \ref{alg:tt_round} is $\mc{O}(Nr^3\underset{i}{\max}~n_i)$ due to the computation of a thin QR and SVD required at each iteration of the first and second for loops, respectively. 
\begin{algorithm}[!ht]
\begin{algorithmic}[1]
  \caption{TT-Round}\label{alg:tt_round}
  \Require Tensor $\ten{X} = [\ten{G}_1, \ten{G}_2, \dots, \ten{G}_{N}]$ in TT format and a tolerance $\epsilon > 0$
  \Ensure TT   approximation $\tenh{X} = [\tenh{G}_{1},\tenh{G}_{2}, \dots, \tenh{G}_{N}]$ with reduced TT-ranks such that it satisfies
  ${\|\ten{X} - \tenh{X}\|_{F}} \le \epsilon{\|\ten{X}\|_{F}}.$
    \State $\delta \gets \frac{\epsilon}{\sqrt{N-1}}\|\ten{X}\|_{F}$
    \LineComment{Right to left orthogonalization}
    \For{$i = N ,\dots, 2$}
        \State $[\mat{Q}, \mat{R}] \gets \textrm{thin-QR}((\mat{G}_{i}^{\{1\}})^{\top})$
    \State $\ten{G}_{i} = \textrm{reshape}(\mat{Q}^{\top},[r_{i-1}, n_i, r_{i}]) $,  $\ten{G}_{i-1} = \ten{G}_{i-1} \times_{3} \mat{R}^{\top}$
    \EndFor
    \LineComment{SVD re-compression}
    \For{$i = 1,\dots, N-1$}
     \State $[\mat{U}, \mat{\Sigma}, \mat{V}] \gets \textrm{truncated-SVD}(\mat{G}_{i}^{\{2\}}, \delta)$ \Comment{$\mat{\Sigma} \in \R^{r_i' \times r_i'}$}
     \State $\ten{\hat{G}}_{i} \gets \textrm{reshape}(\mat{U}, [r_{i-1}', n_i, r_i'])$ \Comment{$r_0'=r_N'=1$}
     \State $\ten{G}_{i+1} = \ten{G}_{i+1} \times_{1} (\mat{\Sigma} \mat{V}^{\top} ) $
    \EndFor
    \State 
    \Return rounded TT-cores $[\tenh{G}_{1}, \tenh{G}_{2}, \dots, \tenh{G}_{N} ]$
    
  \end{algorithmic}
\end{algorithm}

A shortcoming of the TT-SVD algorithm is that one has to explicitly construct the tensor $\ten{X}$. This shortcoming can be alleviated by the TT-cross approximation method, which obtains a TT approximation to a tensor by only sampling a few of its entries. The method has a computational cost of $\mc{O}(Nnr^3)$ flops and given a tolerance $\epsilon_{\text{tol}}$ provides a TT approximation $\tenh{X}$ to $\ten{X}$ that ideally satisfies 
$$ \| \ten{X} - \tenh{X} \|_{C} \le \epsilon \|\ten{X}\|_{C}.$$
We say ideally because the stopping criterion for the TT-cross approximation method involves certain heuristics. We discuss the method in depth in Appendix \ref{sec:tt_cross_approx}. 
Lastly, we present a formula that describes the tensor unfolding matrix $\ten{X}^{\{j\}}$ for $1 \le j \le N$ in terms of its TT-cores, if it admits such a decomposition. Assuming $\ten{X}$ admits an exact TT-representation, the unfolding matrix of the tensor, $\mat{X}^{\{j\}}$ for $1 \le j \le N$ can be expressed in terms of the TT-cores of $\ten{T}$ with the formula \cite[Equation (3.8)]{shi}
\begin{equation} \label{eq:flattening_core_j}
\mat{X}^{\{j\}} = \prod_{i = 1}^{j-1} \left(\mat{I}_{\prod_{k=i+1}^j n_k} \otimes \mat{G}_i^{\{2\}}\right) \mat{G}_j^{\{2\}} \mat{G}_{j+1}^{\{1\}} \prod_{i = j+2}^{N} \left( \mat{G}_i^{\{1\}} \otimes \mat{I}_{\prod_{k = j+1}^{i-1} n_k}\right).
\end{equation}

\subsection{Multivariate Chebyshev Interpolation}\label{ssec:chebyshev}
We review some background on multivariate Chebyshev interpolation. We define  $T_k(x) = \cos(k\cos^{-1}(x))$ to be the Chebyshev polynomial of the first kind of degree $k$, and we then define an invertible linear map, $\varphi_{[\alpha,\beta]}:[-1, 1] \rightarrow [\alpha, \beta]$ that maps the Chebyshev nodes of the first kind on an interval $[-1, 1]$ to the interval $[\alpha, \beta]$ as follows:
$$\varphi_{[\alpha, \beta]}(x) = \frac{\alpha - \beta }{2}x + \frac{\alpha + \beta}{2}.$$
We begin by selecting Chebyshev nodes \( \{\zeta_{j}\}_{j=1}^n\) within the interval \([-1, 1]\) corresponding to the roots of the Chebyshev polynomial $T_n(x)$. Next, we define $\eta_{i} = \varphi_{[a,b]}(\zeta_i)$ for $1 \le i \le n$.
We then define the basis polynomial $\phi^{[a, b]}(\eta_{i}, \cdot)$ on the interval $[a, b]$ as
 \begin{equation}\label{eqn:cheb_trans} \phi^{[a,b]}(\eta_i, x) = \frac{1}{n}  + \frac{2}{n}\sum^{n-1}_{k=1}T_k(\varphi^{-1}_{[a,b]}(\eta_i))T_k(\varphi^{-1}_{[a,b]}(x)).\end{equation}
The interpolating properties of the basis polynomial with respect to the Chebyshev nodes $ \{\eta_{i}\}_{i=1}^{n}$ are demonstrated in the proof of Lemma \ref{lem:cheb_bound}. In particular, we show $\phi^{[a, b]}(\eta_i, \eta_j) = \delta_{i ,j}$ for $1 \le i,j \le n$. We can now use \eqref{eqn:cheb_trans} to define a degree $n-1$  polynomial that interpolates a function $g : [a,b] \rightarrow \R$ on the  nodes $\eta_1, \eta_2, \dots, \eta_n$ as
\begin{equation}\label{eqn:interpolating_polynomial}
\pi_{n-1}(x) = \sum^{n}_{i=1}g(\eta_i ) \phi^{[a, b]}(\eta_i, x).
\end{equation}

This result is perhaps well-known in the literature; here, we restate it so that it will be useful for the later discussion.

\begin{lemma}\label{lem:cheb_bound}
Let $\zeta_1, \zeta_2, \dots, \zeta_n$ be the $n$ Chebyshev nodes of the first kind in the interval $[-1, 1]$. Then
$$\max_{x \in [-1, 1]} \sum^{n}_{k=1}|\phi^{[-1, 1]}(\zeta_k, x)| \le \lambda_{n-1} \le \frac{2}{\pi}\log(n) + 1, $$
where $\lambda_{n-1}$ is the Lebesgue constant; see~\cite[Theorem 15.2]{trefethen2020approximation}.
\end{lemma}
\begin{proof}
For $1 \le i \le n $ define the indicator function $\mathbf{1}_{\zeta_i}:[-1, 1] \rightarrow \R$ such that
\[
\mathbf{1}_{\zeta_i}(t) = 
\begin{cases} 
1 & \text{if } t = \zeta_i\\
0 & \text{otherwise} .
\end{cases}
\]
The basis functions $\{\phi^{[-1,1]}(\zeta_i,x)\}_{i=1}^n$ are interpolants for the nodes $\{\zeta_i\}_{i=1}^n$ which follows from~\cite[Theorem 
 6.7]{mason2003chebyshev}. 
Let $i$ be an arbitrary integer such that $1 \le i \le n$, we can write the interpolant of $\mathbf{1}_{\zeta_i}$ in two different bases as 
$$\sum^{n}_{k=1}\mathbf{1}_{\zeta_i}(\zeta_k)\phi^{[-1, 1]}(\zeta_k, x) = \sum^{n}_{k=1}\mathbf{1}_{\zeta_i}(\zeta_k)\ell_k(x), $$
where $\{\ell_k\}_{k=1}^{n}$  are the Lagrange basis polynomials defined on the Chebyshev nodes $\zeta_1, \zeta_2, \dots, \zeta_n$. Since both bases are polynomials of degree $n-1$ with the same roots, by the fundamental theorem of algebra, they must be the same, i.e., 
$$\phi^{[-1, 1]}(\zeta_i, x) = \ell_i(x),  \quad \forall x \in \R, \> 1 \le i \le n.$$
We now obtain our desired bound:
$$\max_{x \in [-1, 1]} \sum^{n}_{k=1}|\phi^{[-1, 1]}(\zeta_k, x)| \le \max_{x \in [-1, 1]} \sum^{n}_{k=1}|\ell_k(x)| \le \lambda_{n-1} \le \frac{2}{\pi}\log(n) + 1.$$
This completes the proof.
\end{proof}
\paragraph{Multivariate approximation}
For ease of notation, we define
\begin{equation}
    (\alpha_j, \beta_j)= \left\{ \begin{array}{ll}  (\alpha_j^s, \beta_j^s) & 1 \leq j \leq d \\ (\alpha_{j-d}^p, \beta_{j-d}^p) & d+1 \leq j \leq d+d_\theta \\ (\alpha^t_{j- d+ d_\theta}, \beta^t_{j- d + d_\theta}) & d+d_\theta+1 \leq j \leq D.\end{array} \right. \qquad
\end{equation}

Given the Chebyshev nodes $\{\zeta_i\}_{i=1}^n$, we define the set of interpolating points $\{\eta_i^{(j)}\}_{i=1}^{n}$ for $1 \leq j \leq D$ as
\begin{equation}
    \eta^{(j)}_{i} = \varphi_{[\alpha_j, \beta_j]}(\zeta_i).
\end{equation}

We now consider a multivariate Chebyshev approximation of $f_{\kappa}$ \eqref{eqn:fkappa} of the form
\begin{equation}\label{eq:mdc}
f_{\kappa}(\vec\xi) \approx \sum_{i_1=1}^{n} \dots \sum^{n}_{i_D=1}m_{i_1, i_2, \dots, i_D} \prod^{D}_{j=1}\phi^{[\alpha_j, \beta_j]}(\eta^{(j)}_{i_j}, \xi_j),
\end{equation}
where  $\ten{M} \in \R^{n\times \dots \times n}$ is a $D$ dimensional coefficient tensor with each mode size $n$ with entries
$$m_{i_1, i_2, \dots, i_D} = f_\kappa(\vec\eta ), \quad \vec\eta = (\eta^{(1)}_{i_1}, \eta^{(2)}_{i_2},\dots,\eta^{(D)}_{i_D}), \qquad 1 \leq i_j \leq n, 1 \leq j \leq D.$$

\paragraph{Tensor notation} We express the same function approximation using tensor notation. Let \(\vec{x} = (x_1, \dots, x_d) \in \mathcal{X}\) and \( \vec{y} = (y_1, \dots, y_d) \in \mathcal{Y}\) represent arbitrary points from the source and target set. We define the vectorized source and target polynomials, $\vec{q}^{s}_{i}$ and $\vec{q}^{t}_{i}$, for  $1 \le i \le d$ as 
\begin{equation}\label{eqn:q_s_t} \begin{aligned}\vec{q}^{s}_{i}(x_i) = & \> \bmat{\phi^{[\alpha^{s}_{i}, \beta^{s}_{i} ]}(\eta^{(i)}_{1}, x_i) &  \dots & \phi^{[\alpha^{s}_{i}, \beta^{s}_{i} ]}(\eta^{(i)}_{n}, x_i)  } \in \mathbb{R}^{ 1 \times n}, \\
\vec{q}^{t}_{i}(y_i) = & \> \bmat{\phi^{[\alpha^{t}_{i}, \beta^{t}_{i} ]}(\eta^{(d + d_\theta + i)}_{1}, y_i) &   \dots& \phi^{[\alpha^{t}_{i}, \beta^{t}_{i} ]}(\eta^{(d + d_\theta + i)}_{n}, y_i)  } \in \mathbb{R}^{1 \times n}. \end{aligned} \end{equation} Let $\B\theta = (\theta_1, \theta_2, \dots, \theta_{d_\theta}) \in \Theta$, and we define the vectorized polynomial that encodes points in the parameter space as
\begin{equation}\label{eqn:q_p} \vec{q}^{p}_{i}(\theta_i) = \bmat{\phi^{[\alpha^{p}_{i}, \beta^{p}_{i} ]}(\eta^{(d + i)}_{1}, \theta_i) &   \dots & \phi^{[\alpha^{p}_{i}, \beta^{p}_{i} ]}(\eta^{(d + i)}_{n}, \theta_i)  } \in \mathbb{R}^{1 \times n},\end{equation}
for $1\le i \le d_\theta$. With the notation we have established so far, \eqref{eq:mdc} can now be rewritten as 
\begin{equation}\label{eq:mdc2}
    f_{\kappa}(\vec\xi) \approx \ten{M} \times_{i=1}^{d} \vec{q}_{i}^{s}(\xi_i) \times_{i = d+1}^{d + d_\theta} \vec{q}^{p}_{i - d}(\xi_i) \times_{i = d + d_\theta + 1}^{D}\vec{q}_{i - (d + d_\theta) }^{t}(\xi_i) .
\end{equation}

We now discuss how the Chebyshev interpolation of the function $f_\kappa$ translates to a parametric low-rank approximation of the kernel matrix. To this end, we follow the approach in~\cite{fong2009black,Saibaba2022} which we now summarize. 
\paragraph{Parametric low-rank approximation} We now define the factor matrices for the source and target points. The source factor matrices  $\mat{U}_{1}, \dots, \mat{U}_{d} \in \mathbb{R}^{N_s \times n}$ and the target factor matrices $\mat{V}_{1}, \dots, \mat{V}_{d} \in \mathbb{R}^{N_t \times n}$ are defined as
\begin{equation}\label{eqn:uv_factor}\mat{U}_{i} = \begin{pmatrix}
    \vec{q}^{s}_{i}([\vec{x}_1]_{i}) \\
    \vdots \\
    \vec{q}^{s}_{i}([\vec{x}_{N_s}]_{i}) 
\end{pmatrix},  \quad \mat{V}_{i} = \begin{pmatrix}
    \vec{q}^{t}_{i}([\vec{y}_{1}]_{i}) \\
    \vdots \\
    \vec{q}^{t}_{i}([\vec{y}_{N_t}]_{i})
\end{pmatrix} \qquad  \textrm{for }  \> 1 \le i \le d.
\end{equation}
We now introduce the source matrix $\mat{F}_{s} \in \mathbb{R}^{N_s \times n^d}$ and the target matrix $\mat{F}_{t} \in \mathbb{R}^{N_t \times n^d}$ defined as follows using the face-splitting product
\begin{equation}\label{eqn:fs}
\mat{F}_{s} = \mat{U}_{d} \ltimes \mat{U}_{d-1} \ltimes \ldots \ltimes  \mat{U}_{1}, \qquad 
\mat{F}_{t} = \mat{V}_{d} \ltimes \mat{V}_{d-1} \ltimes \ldots \ltimes \mat{V}_{1} .
\end{equation}
It should be noted that the matrices $\mat{F}_s$ and $\mat{F}_t$ are not formed explicitly in our calculations, and we only need to store the factor matrices $\{\mat{U}_j,\mat{V}_j\}_{j=1}^d$. 
We define the parametric tensor $\ten{M}_F(\vec{\theta})$ to be
\begin{equation}\label{eq:ilowrank}
\ten{M}_F(\vec{\theta}) = \ten{M} \times_{i = d+1}^{d + d_\theta} \vec{q}^{p}_{i - d}(
\theta_{i-d}).
\end{equation}
We are now able to iterate over the source and target points in conjunction with \eqref{eq:mdc2}, resulting in the  approximation~\cite[Section 4]{Saibaba2022}
\begin{equation} \label{eqn:paramkernelmatrix}\mat{K}(\mathcal{X},\mathcal{Y};{\vec{\theta}}) \approx \mat{F}_{s}\mat{M}_F^{\{d\}}(\vec{\theta}) \mat{F}_{t}\t.\end{equation}
This approximation is of the form~\eqref{eqn:paramlowrank} but is not feasible for practical scenarios. Observing that the tensor $\ten{M}$ contains $n^{D}$ elements, it should be noted that in practical scenarios, the value of $n^D$ can often become quite large, potentially larger than the product $N_sN_t$. In such a situation, $\mat{K}(\mathcal{X},\mathcal{Y}; \vec{\theta})$ cannot be considered to be low-rank. The Tucker approach proposed in~\cite{Saibaba2022} further compresses the core tensor $\ten{M}$, but it is computationally expensive even for  $d=3$ and $d_\theta = 2$, since $D=8$. To deal with this challenge, we further compress the tensor $\ten{M}$ using the TT format.

\section{Parametric TT Kernel (PTTK) Approximation}
In this section, we present the PTTK method for parametric low-rank approximations. The algorithm is split into two stages, an offline stage which is independent of a specific value of $\vec\theta$, and an online stage in which the low-rank approximation is instantiated for a specific value of $\bar{\vec\theta}$. In Section \ref{ssec:mainalgo}, we go over the details of the  PTTK method by first discussing the offline stage and then the online stage. In Section \ref{ssec:comp_and_sto_cost}, we discuss the computational and storage cost associated with the PTTK method, the main point being the computational cost of the offline stage is linear in $N_s$ and $N_t$ while the computational cost of the online stage is independent of $N_s$ and $N_t$. In Section \ref{ssec:Error_analysis}, we derive an error bound for the low-rank approximation to an arbitrary parametric kernel matrix in terms of the error of the TT-cross approximation and multivariate polynomial interpolation. In Section \ref{sssec:error_disc}, we discuss the error bound given in the previous section, introduce the concept of admissibility, and discuss the smoothness of $f_\kappa$ within the context of the spatial configuration of $\mc{B}_{s}$ and $\mc{B}_{t}$. Lastly, in Section \ref{ssec:Global_case}, we discuss how to apply our PTTK method to the global case ($\mc{X} = \mc{Y}$), where the parametric kernel function is also symmetric and positive definite, and it is desirable to preserve these properties in the approximation. 
\subsection{Main algorithm} \label{ssec:mainalgo}
The algorithm has two stages: an offline stage where a majority of the computations are conducted and is applicable to the entire parameter range $\Theta$, and an online stage where the computations are focused on a specific parameter $\bar{\vec{\theta}}$.

\paragraph{The offline stage} In the offline stage, there are four phases. In Phase 1, we compute the Chebyshev approximation of the function $f_\kappa :\R^D \rightarrow \R$ (here $D = 2d + d_\theta$) and the factor matrices $\{\mat{U}_j\}_{j=1}^d$, $\{\mat{V}_j\}_{j=1}^d$ that define the source matrix $\mat{F}_s$ and the target matrix $\mat{F}_t$, respectively.

In Phase 2, we now compute a TT approximation to the coefficient tensor $\ten{M}$, defined in the previous section, by applying the TT greedy-cross algorithm (Algorithm~\ref{alg:greedy}) on the function $f_\kappa$ which defines the entries of the tensor $\ten{M}$.  We now obtain the following TT approximation
$$\ten{M} \approx \tenh{M} \equiv [\ten{G}_1, \ten{G}_{2}, \dots, \ten{G}_{D}],$$
where $\ten{G}_{i} \in \mathbb{R}^{r_{i-1} \times n \times r_i}$ for $1 \le i \le D$ with the convention $r_0 = r_D = 1$.
The ranks of $\tenh{M}$ depend on the properties of $f_{\kappa}$ such as smoothness, behavior of its derivatives, the domain on which it is defined, etc. 
For more details we refer to Section \ref{sssec:error_disc}. Let $\vec{\theta}\in \Theta$, and we now define the matrix $\mat{H}(\vec{\theta}) \in \mathbb{R}^{r_d \times r_{d + d_\theta}}$ as
\begin{equation}\label{eqn:hmat}\mat{H}(\vec{\theta}) \equiv \mat{H}_1(\theta_1) \cdots \mat{H}_{d_\theta}(\theta_{d_\theta}) \qquad\text{where} \quad \mat{H}_j(\theta_j)\equiv \ten{G}_{d+j} \times_2 \vec{q}^{p}_{j}(\theta_{j}).\end{equation}
Observe that we can obtain a TT-decomposition of the tensor  $\tenh{M}_{F}(\vec{\theta})$ in terms of the TT-cores of $\tenh{M}$ as
\begin{equation}\label{M(theta)_cores}
\tenh{M}_F(\vec{\theta}) =  [\ten{G}_{1}, \dots, \ten{G}_{d}, \ten{T}^{(2)}(\mat{H}(\vec{\theta})), \ten{G}_{d+d_\theta + 1}, \dots, \ten{G}_{D}].
\end{equation}
Here $\ten{T}^{(\cdot)}(\cdot)$ is a twisting tensor function defined in~\eqref{eqn:twist}. The reader should note that $(\ten{T}^{(2)}(\mat{H}(\vec{\theta})))^{\{1\}} = \mat{H}(\vec{\theta})$ and the term represents a degenerate tensor (matrix). We now use unfolding formula  \eqref{eq:flattening_core_j} to obtain an initial approximation to  $\mat{K}(\mathcal{X},\mathcal{Y};{\vec{\theta}})$ in terms of the cores of the tensor $\tenh{M}_F(\vec{\theta}):$
$$\mathat{M}_F^{\{d\}} (\vec{\theta}) = \mat{L}\mat{H}(\vec{\theta})\mat{R}\t, $$ 
where the matrices $\mat{L}$ and $\mat{R}$ are defined implicitly by the formulas (see~\eqref{eq:flattening_core_j})
 \[ \mat{L} = \prod_{i=1}^{d-1}(\mat{I}_{n^{d-i} }  \otimes \mat{G}^{\{2\}}_i ) \mat{G}^{\{2\}}_{d}, \qquad \mat{R}\t = \mat{G}^{\{1\}}_{d + d_\theta + 1} \prod_{i = 1 }^{d - 1}( \mat{G}^{\{1\}}_{d + d_\theta + 1 + i} \otimes \mat{I}_{n^i}). \]
An important point is that the matrices $\mat{L}$ and $\mat{R}$ are not formed explicitly. The summary of the story so far is 
\[ \mat{K}(\mc{X},\mc{Y};{\B\theta}) \approx \mat{F}_s \mat{M}_F^{\{d\}}(\B\theta)\mat{F}_t\t \approx  (\mat{F}_s\mat{L})\mat{H}(\B\theta)(\mat{F}_t\mat{R})\t. \] 
Now Phase 3 can commence which involves computing $\mat{S}\equiv \mat{F}_s \mat{L}$ and $\mat{T} \equiv \mat{F}_t\mat{R}$. These matrices can be computed exactly (in exact arithmetic) through a recursive computation and exploiting the properties of face-splitting/Khatri-Rao products defined in Section \ref{ssec:matrix_operations}, as demonstrated in Phase 3 of the offline stage (Algorithm \ref{alg:offline}). The correctness of Phase 3 in the offline stage can be verified through proof by induction. We illustrate the computation for the case $d = 3$ 
\begin{align*}
    \mat{S}  &=  (\mat{U}_{3} \ltimes \mat{U}_{2} \ltimes \mat{U}_{1})(\mat{I}_{n} \otimes \mat{I}_{n} \otimes \mat{G}_{1}^{\{2\}})(\mat{I}_{n} \otimes \mat{G}_{2}^{\{2\}})\mat{G}_{3}^{\{2\}}\\
      &= (\mat{U}_{3} \ltimes \mat{U}_{2} \ltimes \mat{U}_{1}\mat{G}_{1}^{\{2\}})(\mat{I}_{n} \otimes \mat{G}_{2}^{\{2\}})\mat{G}_{3}^{\{2\}} \\
      &= (\mat{U}_{3} \ltimes ((\mat{U}_{2} \ltimes \mat{U}_{1}\mat{G}_{1}^{\{2\}})\mat{G}_{2}^{\{2\}}))\mat{G}_{3}^{\{2\}}. 
\end{align*}
The computations for  $\mat{T}$ can be approached analogously. We provide a graphical representation of Phase 3 in Figure \ref{fig:offline_mode}. Finally, to obtain a more compressed representation, Phase 4 performs TT rounding (Algorithm \ref{alg:tt_round}) on the tensor in the TT format
\begin{equation*}
 [\ten{T}^{(1)}(\mat{S}), \ten{G}_{d+1}, \ten{G}_{d+2}, \dots, \ten{G}_{d+d_\theta}, \ten{T}^{(3)}(\mat{T}^{\top})].
\end{equation*}
The compressed representations are stored as the matrices $\mat{S}$, $\mat{T}$, and the core tensors $\{\ten{G}_{d+1}, \dots,\ten{G}_{d+d_\theta}\}.$

\begin{figure}
    \centering
    \includegraphics[scale=.75]{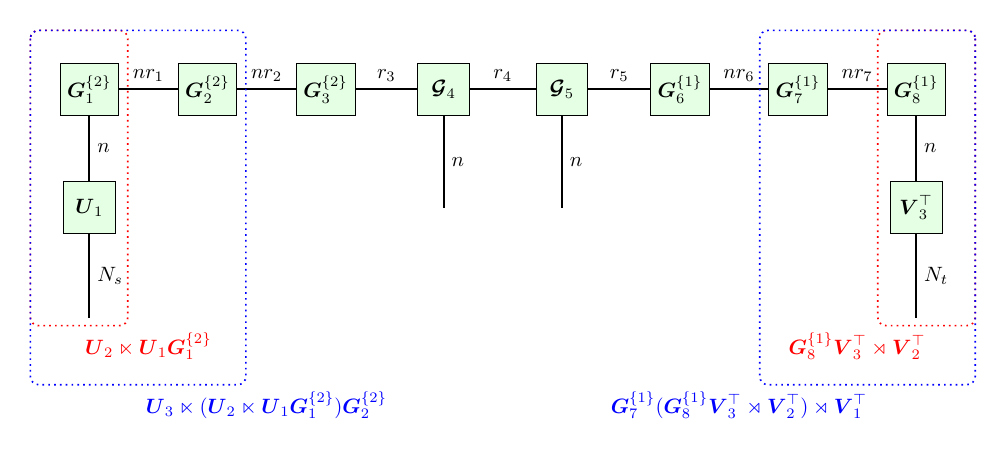}
    \caption{Tensor network diagram of Phase 3 of the offline stage for the case $d=3$ and $d_\theta=2$, where each enclosing dashed rectangle represents the resulting matrix as a tensor network node, obtained by applying the face-splitting/Khatri-Rao product of the proper source/target factor matrix to the consequent matrix obtained after performing the contractions enclosed within the dashed rectangle.}
    \label{fig:offline_mode}
\end{figure}

\begin{algorithm}[!ht]
\begin{algorithmic}[1]
  \caption{PTTK approach: Offline Stage}\label{alg:offline}
  \Require{Source and target points $\mc{X} \subset \mc{B}_s$ and $\mc{Y} \subset \mc{B}_t$, parametric kernel $\kappa(\vec{x},\vec{y}; \B\theta )$ and function $f_\kappa$, input tolerance $\epsilon_{\rm tol}$} 
  \Ensure Matrices $\mat{S}, \mat{T}$, cores $\{\ten{G}_{d+1},\dots,\ten{G}_{d+d_\theta}\}$
    \LineComment{Phase 1: Chebyshev approximation}
    \State Construct matrices $\mat{U}_1,\dots,\mat{U}_d$ and $\mat{V}_1,\dots,\mat{V}_d$ and tensor $\ten{M}$ defining the Chebyshev approximation of the function $f_\kappa$ (see Section~\ref{ssec:chebyshev})
    \LineComment{Phase 2: TT approximation}
    \State Approximate tensor $\ten{M}$ to get tensor $\tenh{M} = [\ten{G}_1,\dots,\ten{G}_D]$  using TT greedy-cross (Algorithm~\ref{alg:greedy}) with input $\epsilon_{\rm tol}$
   \LineComment{Phase 3: Construct factor matrices}
   \State $\mat{S} \gets \mat{U}_{1}\mat{G}_{1}^{\{2\}}$ 
  \For{$i=2,\dots, d $}
   \State $\mat{S} \gets (\mat{U}_{i} \ltimes \mat{S})\mat{G}_{i}^{\{2\}} $
  \EndFor
    \State $\mat{T} \gets \mat{G}_{D}^{\{1\}}\mat{V}_{d}^{\top}$ 
  \For{$i= 1, \dots, d - 1 $}
   \State $\mat{T} \gets \mat{G}_{D - i}^{\{1\}}(\mat{T} \rtimes \mat{V}_{d - i}^{\top}) $
  \EndFor
  \State $\mat{T} \gets \mat{T}^{\top}$
  \LineComment{Phase 4: Additional compression using TT-rounding}
  \State $[\ten{T}^{(1)}({\mat{S}}), \ten{G}_{d+1},  \dots, \ten{G}_{d+d_\theta}, \ten{T}^{(3)}(\mat{T}^{\top})] \gets  $TT-Round(  \\ ~ ~ ~ ~ ~ ~ ~ ~  ~ ~ ~ ~ ~ ~ ~ ~ ~ ~ ~ ~ ~ ~ ~ ~ ~ ~ ~   $[\ten{T}^{(1)}(\mat{S}), \ten{G}_{d+1},  \dots, \ten{G}_{d+d_\theta}, \ten{T}^{(3)}(\mat{T}^{\top})]$,  $\epsilon_{\rm tol}$)
  \State 
   \Return Factor matrices $\mat{S}, \mat{T}$, cores $\{\ten{G}_{d+1},\dots,\ten{G}_{d+d_\theta}\}$\;
  \end{algorithmic}
\end{algorithm}

\paragraph{The online stage} The online stage is fairly straightforward. Given a parameter $\vec{\bar\theta} \in \Theta$, it only requires evaluating the matrix $\mat{H}(\vec{\bar\theta})$ by contracting the cores corresponding to the parameter dimensions (using~\eqref{eqn:hmat}) to obtain the low-rank approximation in~\eqref{eqn:paramlowrank}. The details are given in Algorithm~\ref{alg:online} and the contractions involved for a fixed case are displayed in Figure \ref{fig:online}.

\begin{figure}[!ht]
\includegraphics[scale=.9]{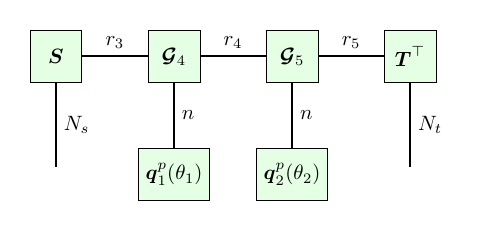}
\centering
\caption{Tensor network diagram of the online stage for fixed dimensions $d = 3$ and $d_\theta = 2$ with a particular parameter $ \vec{\theta} = (\theta_1, \theta_2)$, obtained after performing the operations illustrated in Figure \ref{fig:offline_mode}.}
\label{fig:online}
\end{figure}

\begin{algorithm}[!ht]
\begin{algorithmic}[1]
  \caption{PTTK approach: Online Stage}\label{alg:online}
  \Require{Instance of parameter  $\vec{\bar{\theta}} \in \Theta$, Cores $\{\ten{G}_{d+1},\dots,\ten{G}_{d+d_\theta}\}$} 
  \Ensure{Matrix $\mat{H}(\vec{\bar\theta})$ }
  \State $\mathat{H} = \mat{I}$
  \For{$i = 1 \dots d_\theta$}\State $\mat{H}_{i} := \ten{G}_{d + i} \times_{2} \vec{q}^{p}_{i}(\bar\theta_i)$ and $\mathat{H} \leftarrow \mathat{H} \mat{H}_i$
\EndFor
\State
\Return Core matrix $\mathat{H} \equiv \mat{H}(\vec{\bar\theta})$
  \end{algorithmic}
\end{algorithm}



\subsection{Computational and Storage Cost} \label{ssec:comp_and_sto_cost}
We now give details on the computational and storage costs of Algorithms~\ref{alg:offline} and~\ref{alg:online}.  
\paragraph{Computational Cost} 
Let us first consider the computational cost of the offline stage, as described in Algorithm~\ref{alg:offline}. Forming the factor matrices $\mat{U}_{j}$ and $\mat{V}_{j}$ for each $j$ requires $\mc{O}(n^2(N_s + N_t))$ flops since it requires $\mc{O}(n^2)$ flops to evaluate a vectorized polynomial at a point. In all there are $2d$ factor matrices for an overall cost of $\mc{O}(n^2d(N_s + N_t))$ flops. In Phase 2, the computational cost of Algorithm~\ref{alg:greedy} is $\mc{O}(Dnr^3)$ flops and $\mc{O}(Dnr^2)$ kernel evaluations. With an additional cost of $\mc{O}(Dn^2)$ kernel evaluations and $\mc{O}(Dn^2)$ flops for the initialization (Algorithm~\ref{alg:greedy_init}). Recall that $r$ is the maximum rank of the TT approximation. In Phase 3, forming the products $(\mat{U}_{i} \ltimes \mat{S})\mat{G}_{i}^{\{2\}}$ and $\mat{G}_{D-i}^{\{1\}}(\mat{T} \rtimes \mat{V}_{d-i}^{\top})$ for each $i$ requires $\mc{O}((N_s+N_t)nr^2)$ flops, so the total computational cost of Phase 3 is $\mc{O}((N_s+N_t)dnr^2)$ flops. Lastly, Phase 4 requires applying TT rounding (Algorithm \ref{alg:tt_round}) to the TT tensor,
$$[\ten{T}^{(1)}({\mat{S}}), \ten{G}_{d+1},  \dots, \ten{G}_{d+d_\theta}, \ten{T}^{(3)}(\mat{T}^{\top})], $$
which requires $\mc{O}(d_\theta nr^3 + N_sr^2 + N_tr^2)$ flops. We now consider the computational cost of the online stage, as described in Algorithm \ref{alg:online}. Evaluating the $d_\theta$ vectorized parametric polynomials at a point requires $\mc{O}(d_\theta n^2)$ flops. Performing a mode-2 contraction with tensors $\ten{G}_{d+i}$ for $1 \le i \le d_\theta$ and a $n$ dimensional vectorized parametric polynomial requires $\mc{O}(nr^2)$ flops, and mode-2 contraction must be performed $d_\theta$ times, which gives us an overall computational cost of $\mc{O}(d_\theta (nr^2 + n^2))$ flops. Then forming the matrix $\mathat{H}$ requires $\mc{O}(d_\theta r^3)$ flops. The overall computational cost for the offline stage is $\mc{O}( d(nr^2 + n^2)(N_s + N_t) + Dnr^3 + Dn^2 )$ flops and $\mc{O}(Dn^2 + Dnr^2)$ kernel evaluations. The overall computation cost for the online stage is 
$\mc{O}(d_\theta (n^2 + nr^2 + r^3))$ flops. An important point here is that only the offline stage requires kernel evaluations and has a computational cost dependent on $N_s$ and $N_t$; the cost of the online stage does not have any dependence on $N_s$ and $N_t$.

\paragraph{Storage Cost}
The storage cost of our PTTK method is determined by the cost to store the matrices $\mat{S}, \mat{T}$ and TT-cores $\{\ten{G}_{d+1}, \ten{G}_{d+2}, \dots, \ten{G}_{d+d_\theta}\}$. Therefore, the storage cost of our PTTK method is
$N_sr_{d} + N_tr_{d+d_\theta} + \sum_{j=1}^{d_\theta} nr_{d + j - 1}r_{d + j } $ floating point numbers.

\subsection{Error analysis} \label{ssec:Error_analysis}
There are two sources of error in the PTTK approximation: the error due to the Chebyshev approximation and the error due to the TT approximation. We make assumptions on the smoothness of $f_\kappa$ (see~\eqref{eqn:fkappa}) and the error in the TT approximation. Let $\Xi = \mc{B}_s \times \Theta \times \mc{B}_t  \subset \R^D$ be the domain of $f_\kappa$, which we can write as the hypercube $\Xi =  \bigtimes_{i=1}^D [\alpha_i,\beta_i ]$, and we define  $$\textrm{diam}_{\infty}(\Xi) \equiv \max \{\beta_i - \alpha_i: 1 \le i \le D  \}.$$
 We assume $f_\kappa \in C^{\infty}(\Xi)$, and there exist $ C_f  \ge 0 , \gamma_{f} > 0$, and $\sigma \in \mathbb{N}$ such that the following bound
 \begin{equation} \label{eq:regularity_bound}
     \sup_{\vec{x} \in \Xi} | \partial^{v}_{i} f_\kappa(\vec{x}) | \le \frac{C_f(v + \sigma - 1)!}{\gamma_f^v(\sigma-1)!} \quad \textrm{holds for all integers   }  1 \le i \le D, v \ge 0. 
\end{equation}
Additionally, for integers $k \ge 1 $ we define the following term 
 \[ E(k) \equiv 2eDC_{f}(\lambda_{k-1}+1)^D (k+1)^\sigma \left(1 + \frac{\textrm{diam}_{\infty}(\Xi)}{\gamma_f}\right)\varsigma\left(\frac{2 \gamma_f}{\textrm{diam}_{\infty}(\Xi)}\right)^{-k}
\] 
where the function $\varsigma$ is defined to be $\varsigma:r \mapsto r + \sqrt{1 + r^2}$.
\begin{proposition}\label{prop:error}
 Consider the notation and assumptions established in this section. 
The error in the PTTK approximation satisfies
\[
\max_{\theta \in \Theta} \|\mat{K}(\mc{X}, \mc{Y};{\vec{\theta}}) - \mat{S}\mat{H}(\vec{\theta})\mat{T}^{\top}\|_C \le E(n) +  \lambda_{n-1}^D \| \ten{M} - \tenh{M}\|_{C},
\]
where $\tenh{M}$ is the TT approximation of $\ten{M}$. Recall $\lambda_{n-1}$ indicates the Lebesgue constant of the Chebyshev nodes.
\end{proposition}

In the bound above, the error has two components $E(n)$ which represents the error in the Chebyshev approximation and the second term $\lambda_{n-1}^D \|\ten{M} - \tenh{M}\|_C$ that represents the error in the TT approximation which is amplified by the interpolation through the factor $\lambda_{n-1}^D$. The factor $E(n)$ shows exponential decay with the number of Chebyshev points $n$, whereas the factor $\lambda_{n-1}^D$ shows polylogarithmic growth in $n$.

\begin{proof}[Proof of Proposition~\ref{prop:error}]

Using the expression $\mat{SH}(\vec\theta)\mat{T}\t = \mat{F}_s\mathat{M}_F^{\{d\}}(\vec\theta)\mat{F}_t\t $ and the triangle inequality, we obtain
    \begin{equation}\label{eqn:inter0}
\begin{aligned}
     \|\mat{K}(\mc{X},\mc{Y};{\vec\theta}) - \mat{SH}(\vec\theta)\mat{T}\t\|_C \leq & \> \underbrace{\|\mat{K}(\mc{X},\mc{Y};{\vec\theta}) -  \mat{F}_s \mat{M}_F^{\{d\}}(\vec\theta)\mat{F}_t\t \|_C}_{\equiv \alpha} ~ +\\
     & \> \qquad \underbrace{\|\mat{F}_s \mat{M}_F^{\{d\}}(\vec\theta)\mat{F}_t\t -  \mat{F}_s\mathat{M}_F^{\{d\}}(\vec\theta)\mat{F}_t\t\|_C}_{\equiv \beta}  .
\end{aligned}\end{equation}
We deal with both the terms separately. 

We tackle the first term $ \alpha$. By the assumptions on $f_\kappa$ and as a consequence of Corollary 4.21 in \cite{borm2010efficient}, the error in the Chebyshev norm satisfies
\begin{equation}\label{eqn:inter1} \max_{\vec\theta\in \Theta} \| \mat{K}(\mc{X},\mc{Y};{\vec\theta}) - \mat{F}_s \mat{M}^{\{d\}}(\vec\theta)\mat{F}_t\t \|_C \leq  E(n). \end{equation}

We tackle the second term $\beta$. Fix a point $\vec\xi = (\vec{x},\vec\theta,\vec{y}) \in \mc{X} \times \Theta \times \mc{Y}$ at which the Chebyshev norm of the term $\beta$ is acquired. For this point $\vec\xi$, the corresponding entry of $\mat{F}_s\mat{M}_F^{\{d\}}(\vec\theta)\mat{F}_t\t$ can be expressed as $\ten{M} \times_{j=1}^D \vec\phi_j(\xi_j)$, where the vectors of polynomials $\vec\phi_j(\cdot)$ are defined as (see~\eqref{eqn:q_s_t} and~\eqref{eqn:q_p})
\[ \vec\phi_j(\xi_j) = \left\{ \begin{array}{cc} \vec{q}_j^s(\xi_j) & 1 \le j \le d \\ \vec{q}_{j-d}^p(\xi_j) & d+1 \le j \le d+d_\theta \\ \vec{q}^t_{j-(d+d_\theta)}(\xi_j) & d+d_\theta+1 \le j \le D.  \end{array}  \right. \] 
Similarly, with the notation just defined, the corresponding entry of $\mat{F}_s\mathat{M}_F^{\{d\}}(\vec\theta)\mat{F}_t\t$ can be expressed as $\tenh{M} \times_{j=1}^D \vec\phi_j(\xi_j)$. Therefore, we can bound the error $\beta$ as 
\[ \begin{aligned}\|\mat{F}_s \mat{M}_F^{\{d\}}(\vec\theta)\mat{F}_t\t -  \mat{F}_s\mathat{M}_F^{\{d\}}(\vec\theta)\mat{F}_t\t\|_C = & \> |(\ten{M}- \tenh{M}) \times_{j=1}^D\vec{\phi}_j(\xi_j) |\\ 
\le 
  & \>  |\sum_{i_1,\dots,i_D} (\ten{M}- \tenh{M})_{i_1,\dots,i_D} \prod_{j=1}^D[\vec{\phi}_j(\xi_j)]_{i_j} |, \\
  \le & \>  \| \ten{M}- \tenh{M}\|_C |\sum_{i_1,\dots,i_D}  \prod_{j=1}^D |[\vec{\phi}_j(\xi_j)]_{i_j}| .
\end{aligned}\]

Next, we use the equality $\sum_{i_1,\dots,i_D}  \prod_{j=1}^D |[\vec{\phi}_j(\xi_j)]_{i_j}|  = \prod_{j=1}^D \left(\sum_{i_j=1}^n |[\vec{\phi}_j(\xi_j)]_{i_j}| \right)   $ to write
\begin{equation}\label{eqn:inter2} \begin{aligned} 
\|\mat{F}_s \mat{M}_F^{\{d\}}(\vec\theta)\mat{F}_t\t -  \mat{F}_s\mathat{M}_F^{\{d\}}(\vec\theta)\mat{F}_t\t\|_C \leq & \> \|\ten{M}-\tenh{M}\|_C   \prod_{j=1}^D \|\vec{\phi}_j(\xi_j)\|_1 \\
\leq & \> \|\ten{M}-\tenh{M}\|_C (\lambda_{n-1})^D.  \end{aligned}\end{equation} 

The inequality $\|\vec{\phi}_j(\xi_j)\|_1 \le 
\lambda_{n-1}$ follows from Lemma \ref{lem:cheb_bound}. The stated result follows by plugging the intermediate results~\eqref{eqn:inter1} and~\eqref{eqn:inter2} into~\eqref{eqn:inter0}.
\end{proof}

The results of \cref{prop:error} assume the existence of a TT approximation $\tenh{M}$. In this next result, we show that such an approximation can be found.

\begin{proposition}\label{prop:compress} 
Consider the notation and assumptions established in this section.
For an integer $p$ s.t. $1 \le p \le n$, there exists a TT decomposition of $\ten{M}$, which we label as $\tenh{M}$ with TT  ranks $r_{k+1} \le \min\{p^{k}, p^{D-k} \}$ for $1 \le k < D -1$ ($r_1 = r_D = 1$), that satisfies
$$\| \ten{M} - \tenh{M}\|_{C} \le E(p, n),$$
whenever $ p < n$ and $\| \ten{M} - \tenh{M}\|_{C} = 0$ if $p = n$. 
 
\end{proposition}
\begin{proof}
The proof strategy is based on~\cite{Shi-Townsend}. 
The idea is to construct a multi-variate Chebyshev polynomial $g$ that interpolates $f_{\kappa}$ at the $p^D$ Chebyshev nodes defined on $\Xi$ using the method outlined in Section \ref{ssec:chebyshev}. We define the tensor $\tenh{M}$ with entries
$$[\tenh{M}]_{i_1, i_2, \dots, i_D} = g(\eta^{(1)}_{i_1}, \eta^{(2)}_{i_2}, \dots, \eta^{(D)}_{i_D}), \quad 1 \le j \le D, 1 \le i_j \le p.$$
Note, the Chebyshev nodes $\{\eta^{(j)}_{i_j} \}_{i_j =1}^{p}$ for $1 \le j \le D$ are associated with the Chebyshev polynomial $g$.
Then, by Lemma 3.1 and (3.3) from \cite{Shi-Townsend}, we can conclude there exists an exact TT decomposition of $\tenh{M}$,
$$\tenh{M} = [\ten{C}_1, \ten{C}_2, \dots, \ten{C}_{D}],$$
with TT ranks $r_{k} \le  \min\{p^{k}, p^{D -k} \}$ for $1 \le k < D$ such that $\|\ten{M} - \tenh{M} \|_C \le \|f_{\kappa} - g \|_{\infty}$. Next, by the assumptions on $f_\kappa$ and as a consequence of Corollary 4.21 in \cite{borm2010efficient}, we have
$$\|\ten{M} - \tenh{M} \|_C \le \|f_{\kappa} - g \|_{\infty} \le  E(p).$$
 If $p = n$ then the polynomial $g$ interpolates $f_\kappa$ exactly, and $\ten{M} = \tenh{M}$. 
\end{proof}

\subsection{Discussion of error analysis} \label{sssec:error_disc}
We have demonstrated that the TT ranks of $\tenh{M}$ are bounded, and have derived a relationship between the TT ranks of $\tenh{M}$ and the derivatives of $f_\kappa$ (see \eqref{eq:regularity_bound}). Note the bound is quite pessimistic concerning the TT ranks, and this is to be expected since we did not assume additional structure on $f_{\kappa}$. The existence of an approximation of $f_{\kappa}$ into a summation of separable basis functions allows us to derive a relationship between the TT ranks and the TT approximation error. In the context of Proposition~\ref{prop:compress}, this approximation was obtained via multi-variate Chebyshev interpolation. The following \cite{DBLP:journals/corr/WangLD17,randomfeatures-1, cotter2011explicit, greengard1991fast} discuss approximations of specific classes of kernel functions into a summation of separable basis functions.

We now introduce the concept of admissibility. We first define the distance between $\mathcal{B}_{s}$ and $\mathcal{B}_{t}$:
$$\textrm{dist}(\mathcal{B}_{s}, \mathcal{B}_{t}) = \min_{\vec{x} \in \mathcal{B}_{s}, \vec{y} \in \mathcal{B}_{t}} \|\vec{x} - \vec{y}\|_{2}.$$ 
We say that $\mc{B}_{s}$ and $\mc{B}_{t}$ are \emph{admissible} whenever $\textrm{dist}(\mc{B}_{s}, \mc{B}_{t}) > 0$ and we say that $\mc{B}_{s}$ and $\mc{B}_{t}$ are \emph{weakly admissible} whenever $\mc{B}_{s}$ and $\mc{B}_{t}$ intersect at a single point. This definition for $\mc{B}_{s}$ and $\mc{B}_{t}$ being admissible is chosen because many kernels (for example, the Bi-harmonic kernel) have singularity at the origin. See Table~\ref{tab:kern-table} for a listing of kernels used in the numerical experiments section and their definitions.

Let us first consider the kernel functions $\kappa$ that are nonparametric (i.e. $d_\theta = 0$).  If $\kappa$ is asymptotically smooth (Definition 4.14 from \cite{Hackbusch2015}) and $\mc{B}_{s}$ and $\mc{B}_{t}$ are admissible, then $f_\kappa \in C^\infty(\Xi)$,  by Lemma E.5 from \cite{Hackbusch2015}. 
This means the Bi-harmonic, Laplace-3D, Laplace-2D, and Thin-plate kernels, which are known to be asymptotically smooth, result in $f_\kappa \in C^{\infty}(\Xi)$ if $\mc{B}_{s}$ and $\mc{B}_{t}$ are admissible.

We now turn to the parametric case. We consider the special case that $\kappa$ is an isotropic kernel with just the length scale parameter; as in, $\vec{\theta} = (\ell)$ s.t. $\ell > 0 $  and $$\kappa(\vec{x},\vec{y}; \ell) = g\left(\frac{\|\vec{x} - \vec{y}\|_{2}}{\ell}\right), $$   where $g:\R_+ \rightarrow \R$. This special case is of importance, since it covers the majority of parametric kernels considered in this paper. We also first consider the admissible case ($\text{dist}(\mc{B}_{s}, \mc{B}_{t}) > 0$). If $g$ is in $C^{\infty}(\R_{+})$ then  $f_{\kappa} \in C^{\infty}(\Xi)$ by Theorem 3.1 from \cite{edwards2012advanced}. All the parametric kernels we consider, such as the Exponential, Mat\'ern (with a fixed $\nu$), and Thin-Plate spline, can be expressed in the isotropic form where $g \in  C^{\infty}(\R^{+})$. Hence, many of the parametric kernels we consider result in  $f_\kappa \in C^{\infty}(\Xi)$ whenever $\textrm{dist}(\mc{B}_{s}, \mc{B}_{t}) > 0.$ Even if $\mc{B}_s$ and $\mc{B}_t$ are not admissible, some kernels such as the Squared-Exponential and Multiquadric kernel result in $f_k \in C^\infty(\Xi)$.

\subsection{Global low-rank approximation} \label{ssec:Global_case}
In this section, we consider the case that the source and target boxes ($\mc{B}_s$ and $\mc{B}_t$) are the same and assume that the source and target points ($\mc{X}$ and $\mc{Y}$) are the same. If the matrix kernel function $\kappa(\vec{x},\vec{y};{\vec\theta}) = \kappa(\vec{y},\vec{x};{\vec\theta})$ is symmetric in its arguments, then the low-rank approximation should be symmetric, and if the kernel function is symmetric positive definite, then the approximation ought to also be symmetric positive semi-definite. The algorithms described in Section~\ref{ssec:mainalgo} currently do not consider the problem's symmetry nor the kernel function being positive definite when $\mc{X} = \mc{Y}$. However, we can modify them appropriately, which we will now discuss. 
\paragraph{Method}
Given a matrix $\mat{A} \in \R^{ m \times m}$, its symmetric part $\frac12 (\mat{A} + \mat{A}\t)$ is the best symmetric approximation to $\mat{A}$~\cite[Theorem IX.7.1]{bhatia1997matrix}. Therefore, to obtain a symmetric low-rank approximation, we can work with the symmetric part, i.e.,
\begin{equation}\label{eqn:pttk1} \frac12\left(\mat{SH}(\vec{\bar \theta})\mat{T}\t + \mat{T}(\mat{H}(\vec{\bar \theta}))\t\mat{S}\t\right) = \bmat{\mat{S} & \mat{T}} \mathat{H} \bmat{\mat{S}\t \\ \mat{T}\t},  \end{equation}
where $\mathat{H}$ is a block matrix defined as 
\[  \mathat{H} \equiv \bmat{ & \frac12\mat{H}(\vec{ \bar\theta}) \\ \frac12 (\mat{H}(\vec{ \bar\theta}))\t  & }.\]
Therefore, this low-rank representation can be obtained without any additional storage requirements, but the downside is that the rank of the low-rank approximation is doubled. We now take the thin QR decomposition of the block matrix composed of $\mat{S}$ and $\mat{T}$ as
$ \bmat{\mat{S} & \mat{T}} = \mat{QR}.$
We designate this point of the method as the end of the offline stage because the matrices $\mat{Q}$ and $\mat{R}$ can be pre-computed and stored.

\begin{algorithm}[!ht]
\caption{}
\begin{algorithmic}[1]
  \caption{PTTK-Global: Offline and Online Stages}\label{alg:pttk_global}
  \Require{Source points $\mc{X} \subset \mc{B}_s$ , positive definite parametric kernel $\kappa(\vec{x},\vec{y}; \B\theta)$ and function $f_\kappa$, input tolerance $\epsilon_{\rm tol}$, Compress (true/false)}
  \Ensure{Matrices $\mat{Q}, \mat{W}$, defining the low-rank approximation $\mat{QWQ}\t$}
  \LineComment{Offline stage}
    \State Compute $\mat{S}, \mat{T}, \{\ten{G}_{d+1},\dots,\ten{G}_{d+d_\theta}\} $ using Algorithm \ref{alg:offline} with the inputs above and additionally with $\mc{X} =\mc{Y}$ and $\mc{B}_s = \mc{B}_t$.
    \State Compute thin-QR $ \mat{Q}, \mat{R} \gets \textrm{thin-QR}([\mat{S} , \mat{T}])$
  \LineComment{Online stage}
\State Compute $\mat{H}(\vec{\bar{\theta}})$ using Algorithm~\ref{alg:online} with inputs  $\{\ten{G}_{d+1},\dots,\ten{G}_{d+d_\theta}\}$
\State $\mat{U}, \mat{\Lambda} \gets \textrm{eigendecomposition}(\mat{R}   \mathat{H}\mat{R}^{\top})$, with  $\mat{\Lambda} =  \textrm{diag}(\lambda_1, \lambda_2, \dots, \lambda_{k})$ satisfies $|\lambda_1| \ge |\lambda_2| \ge \dots \ge |\lambda_{k}|$, and where $\mathat{H}$ is defined in~\eqref{eqn:pttk1}
\State Zero out the negative eigenvalues (if any) in $\mat{\Lambda}$ and discard the corresponding eigenvectors to obtain matrices $\mat{\Lambda}$ and $\mat{U}$
\If{Compress}
\State Further zero out small eigenvalues of $\mat{\Lambda}$ to obtain $\mathat{\Lambda}$ and discard the corresponding eigenvectors to obtain $\mathat{U}$ such that 
$${\| \mat{Q}\mathat{U} \mathat{\Lambda} (\mat{Q}\mathat{U})^{\top} - \mat{S}\mat{H}(\vec{\bar \theta })\mat{T}\|_{F}}\le  \epsilon_{\rm tol} \| \mat{S}\mat{H}(\vec{\bar \theta })\mat{T} \|_{F}.$$
\State Set $\mat{Q} \gets \mat{Q}\mathat{U}$, $\mat{W} \gets \mathat{\Lambda}$
\Else
\State Set $\mat{W} \gets \mat{U}\mat{\Lambda}\mat{U}^{\top}$
\EndIf 
\State 
\Return $\mat{Q}$ and $\mat{W}$ 
  \end{algorithmic}
\end{algorithm}
The online stage now begins  by instantiating a  particular parameter $\vec{\bar \theta} \in \Theta$. Using the parameter $\vec{\bar \theta}$, we take the eigendecomposition of the symmetric matrix $ \mat{R} \mathat{H}  \mat{R}^{\top} = \mat{U\Lambda U}\t.$
It is assumed that the kernel function is positive definite; thus,  we zero out the negative eigenvalues and discard the small eigenvalues and the appropriate eigenvectors. From discarding the relevant eigenvalues along with the appropriate eigenvectors, we obtain a new pair of induced matrices $\mathat{\Lambda}$ and $\mathat{U}$. We can now define the matrix $\mat{Z} = \mat{Q}\mathat{U}$, and we now have the low-rank factorization defined by the expression $\mat{K}(\mathcal{X}, \mathcal{X};{\vec{\bar\theta}})\approx \mat{Z} \mathat{\Lambda}\mat{Z}\t$. Note, the additional compression afforded by this low-rank factorization. Alternatively, we can form the matrix $\mat{A} = \mathat{U}\mathat{\Lambda}\mathat{U}^{\top}$ and obtain a low-rank factorization:
\begin{equation}\label{eqn:uncompressed_form}
\mat{K}(\mathcal{X}, \mathcal{X};{\vec{\bar\theta}}) \approx \mat{Q}\mat{A}\mat{Q}^{\top}.
\end{equation}
The advantage of the alternative low-rank factorization is that $\mat{A}$ is the only matrix in the expression dependent on $\vec{\bar\theta}$; however, we lose out on additional compression afforded by the initial low-rank factorization. We formalize this discussion in Algorithm~\ref{alg:pttk_global}.
Additionally, we refer to the rank in \eqref{eqn:uncompressed_form} as the \textit{uncompressed rank}. Lastly, in the numerical experiments section, we label the \textbf{PTTK-Global} method as \textbf{PTTK-Global-1} when we pass in false for the ``Compress'' input variable for Algorithm ~\ref{alg:pttk_global}. Analogously, we label the method as \textbf{PTTK-Global-2} when we pass in true for the ``Compress'' input variable for Algorithm~\ref{alg:pttk_global}.
\paragraph{Additional Computational Cost}
Along with the computational costs associated with the offline and online stages mentioned in Section \ref{ssec:mainalgo}, we invoke additional computational costs with Algorithm~\ref{alg:pttk_global}, which we will now discuss. During the offline stage, we must now form the block matrix composed of matrices $\mat{S}$ and $\mat{T}$ and then take its QR decomposition; this step has a computational cost of $\mc{O}(N_sr^2)$ flops. For the online stage, forming the matrix $\mathat{H}$ then subsequently taking the eigendecomposition of $\mat{R}\mathat{H}\mat{R}^{\top}$ and editing the eigenvalues of $\mat{\Lambda}$ along with discarding the corresponding eigenvectors of $\mat{U}$ has a computational cost of $\mc{O}(d_\theta n r^2 + d_\theta r^3)$ flops. For {PTTK-Global-1}, forming the matrix $\mat{U}\mat{\Lambda}\mat{U}^{\top}$ has a computational cost of $\mc{O}(r^3)$ flops. For {PTTK-Global-2}, further zeroing out the eigenvalues of $\mat{\Lambda}$ and discarding the corresponding eigenvalues of $\mat{U}$, then forming the matrix $\mat{Q}\mathat{U}$ has a computational cost of $\mc{O}(N_sr^2)$ flops. To summarize, for {PTTK-Global-1} and PTTK-Global-2, the offline stage has the computational cost associated with Algorithm~\ref{alg:offline} and an additional  $\mc{O}(N_sr^2)$ flops, and for the online stage there is a cost of $\mc{O}(d_\theta r^3 + d_\theta n r^2)$  flops for {PTTK-Global-1} and $\mc{O}(d_\theta r^3 + d_\theta n r^2 +N_sr^2)$  flops for {PTTK-Global-2}.

\section{Numerical Experiments}
All numerical experiments were run on the following computer: OS: Ubuntu 22.04, CPU: Intel Core i7-7600U @ 4x 3.9GHz, RAM: 16GB, SWAP: 16GB.
We utilize our PTTK and TTK (see Section \ref{ssec:TTK_Method}) low-rank approximation methods to approximate various parametric and non-parametric kernels, respectively. Table \ref{tab:kern-table} enumerates the kernels utilized for subsequent low-rank approximation. Denote the norm by \( \|\cdot\|_p \), where \( p \in \{2,F\} \) is a parameter representing the type of norm used (either spectral or Frobenius). The relative error is then computed as:
\[ \textrm{relative error}\equiv  \frac{\| \mathbf{K}(\mathcal{X},\mathcal{Y};{\vec{\theta}}) - \hat{\mathbf{K}}(\mathcal{X},\mathcal{Y};{\vec{\theta}}) \|_p}{\| \mathbf{K}(\mathcal{X},\mathcal{Y};{\vec{\theta}}) \|_p}. \]
When the matrices are too large to be formed explicitly, we utilize random uniform sampling of points from both $\mathcal{X}$ and $\mathcal{Y}$ to construct the sampled source and target points, denoted by $\mathcal{\hat{X}}$ and $\mathcal{\hat{Y}}$, respectively. Then, we compute the relative error for the subsampled kernel matrix, which is expressed as:
\[ \textrm{subsampled relative error} \equiv  \frac{\| \mathbf{K}(\mathcal{\hat{X}},\mathcal{\hat{Y}}; \vec \theta ) - \hat{\mathbf{K}}(\mathcal{\hat{X}},\mathcal{\hat{Y}}; \vec \theta ) \|_p}{\| \mathbf{K}(\mathcal{\hat{X}},\mathcal{\hat{Y}}; \vec \theta ) \|_p}. \]
Lastly, we refer to the \textit{compression rank} of a low-rank kernel matrix approximation method as the target rank given by that method. The target rank can either be user supplied or computed by that method based on a user supplied error tolerance.

\begin{table}[!ht]
    \centering
    \begin{tabular}{l|l}
        \textbf{Name} & \textbf{Kernel Function} \\
        \hline
        Bi-harmonic & $r^{-2}$ \\
        Laplace-3D & $r^{-1}$ \\
        Laplace-2D & $-\log(r)$ \\
        Thin-plate & $r^2\log(r)$ \\
        Thin-plate spline & $\frac{r^2}{\ell^2} \log(\frac{r^2}{\ell^2} )$ \\
        Squared-Exponential (SE) & $\exp\left(-\left(\frac{r}{\ell}\right)^2\right)$ \\
        Multiquadric & $(1 + \left(\frac{r}{\ell}\right)^2)^{1/2}$ \\
        Exponential & $\exp\left(-\frac{r}{\ell}\right)$ \\ 
        Mat\'ern & $\frac{2^{ (1-\nu)}}{\Gamma(\ell)}\left(\sqrt{2\nu}\frac{r}{\ell}\right)^{\nu} B_{\nu}\left(\sqrt{2\nu}\frac{r}{\ell}\right)$ 
    \end{tabular}
    \caption{Kernel functions of the form  $\kappa(\vec{x},\vec{y}; \vec \theta )$,  for two different types of parameterization $\B\theta = (\ell, \nu)$ and $\B\theta = (\ell)$. The vectors $\vec{x} \in \mathcal{X}$ and $\vec{y} \in \mathcal{Y}$ with the pairwise distance $r = \| \vec{x} - \vec{y}\|_{2}$ and $B_{\nu}$ is the modified Bessel function of the second kind.}
        \label{tab:kern-table}

\end{table}
We select kernel functions from different application domains to demonstrate the versatility of our PTTK method. The kernels Laplace-2D, Laplace-3D, and biharmonic appear in applications related to PDEs and integral equations; SE, multiquadric, Thin-plate, and Thin-plate spline appear in radial basis function interpolation; SE and Exponential are both special cases of Mat\'ern and appear in Gaussian processes. Note that all the chosen kernels are symmetric and isotropic. The listed kernels can have varying smoothness; in particular, many of them exhibit singularity near the origin, see Section \ref{sssec:error_disc}. In the numerical experiments, we choose $\mc{B}_{s}$ and $\mc{B}_{t}$ so that the parametric kernel matrices admit low-rank approximations based on the  smoothness of the chosen parametric kernel function. The numerical experiments in this section are implemented in Python; the code to reproduce the tables is located at \url{https://github.com/awkhan3/ParametricTensorTrainKernel}. 
\subsection{Non-parametric Kernels}
\label{ssec:TTK_Method}
Before we consider the parametric kernels, we present a simple extension to non-parametric kernels to better understand the accuracy. We call this method, the TTK (Tensor Train Kernel) method.
\paragraph{TTK method} The TTK method can be expressed as a special case of the 
PTTK method where $\mat{H}(\vec{\theta}) = \mat{I}_{r_d}$ is the identity matrix of size $r_d \times r_d$ and $d_\theta = 0$; thus, the online stage (Algorithm \ref{alg:online}) is not needed, and the low-rank approximation of $\mat{K}(\mathcal{X}, \mathcal{Y})$ can simply be expressed in terms of the matrices  $\mat{S}$ and $\mat{T}$ obtained from the offline stage (Algorithm \ref{alg:offline}), resulting in  a low-rank factorization of $\mat{K}(\mathcal{X},\mathcal{Y})$: 
$$\mat{K}(\mathcal{X},\mathcal{Y}) \approx \mat{S}\mat{T}\t.$$
For this experiment, we fix the dimension to be $d = 3$ and take $n = 27$ Chebyshev nodes in each spatial dimension. Other settings are explained below:
\begin{enumerate}
    \item \textbf{Boxes}: We define the source box $\mathcal{B}_{s} = [0, 1]^{3}$ and the target box $\mathcal{B}_{t} = [2, 3]^{3}$ such that $\mathcal{X} \subseteq \mathcal{B}_{s}$ and $\mathcal{Y} \subseteq \mathcal{B}_{t}$. We take $N_s = N_t = 10^4$ points for the source and target points in a uniform random manner within the boxes $\mathcal{B}_{s}$ and $\mathcal{B}_{t}$, respectively. 
    \item \textbf{Kernels considered}: We consider all the kernels listed in Table~\ref{tab:kern-table}. For all the kernels, we fix the correlation length $\ell = 1$, and for the Mat\'ern kernel, we fix the smoothness parameter $\nu$ to $\frac{3}{2}$ and $\frac{5}{2}$ which we designate as Mat\'ern-3/2 and Mat\'ern-5/2, respectively.
\end{enumerate} For all of these experiments, we fix the tolerance given to Algorithm~\ref{alg:offline}
 to be \num{1e-9}.
To compute the \textbf{SVD-Error}, we compute the thin SVD of the kernel and truncate the rank at the compression rank (designated by the column \textbf{rank}) obtained by TTK. This ensures that both the SVD and TTK have the exact storage cost. The results are reported in Table~\ref{tab:TTK-table}. 
\begin{table}[!ht]
\small\centering
\begin{tabular}{c|c|c|c|c|c}\hline
\textbf{Kernel} & \textbf{rank} & \begin{tabular}{@{}l@{}}
                   \textbf{TT}\\
                   \textbf{Error}\\
                 \end{tabular}& \begin{tabular}{@{}l@{}}
                   \textbf{SVD}\\
                   \textbf{Error}\\
                 \end{tabular} & \begin{tabular}{@{}l@{}}
                   \textbf{TT}\\
                   \textbf{Time (s)}\\
                 \end{tabular} & \begin{tabular}{@{}l@{}}
                   \textbf{SVD}\\
                   \textbf{Time (s)}\\
                 \end{tabular} \\
\hline
Exponential & 58 & \num{5.07e-10} & \num{4.27e-11} & 2.86  & 331  \\
Thin-plate & 43 & \num{3.85e-10} & \num{4.64e-11} & 2.89  & 332  \\
Bi-harmonic & 65 & \num{3.29e-10} & \num{7.85e-11} & 3.26  & 332  \\
Multiquadric & 42 & \num{1.73e-10} & \num{4.78e-11} & 3.08  & 332  \\
Thin-plate-spline & 43 & \num{3.25e-10} & \num{4.64e-11} & 2.95  & 331  \\
Laplace-2D & 46 & \num{7.65e-11} & \num{3.25e-11} & 3.07  & 331  \\
Laplace-3D & 43 & \num{9.93e-11} & \num{3.58e-11} & 3.00   & 332  \\
Mat\'ern-3/2 & 62 & \num{1.03e-9} & \num{5.57e-11} & 3.17  & 332  \\
Mat\'ern-5/2 & 65 & \num{1.04e-9} & \num{6.62e-11} & 3.22  & 331  \\
SE & 91 & \num{1.41e-08} & \num{2.97e-10} & 3.35  & 331  
\end{tabular}
\caption{TTK applied to various kernels defined in Table \ref{tab:kern-table}, with a tolerance fixed at \num{1e-9}. The columns \textbf{TT-Error} and \textbf{SVD-Error} represent the relative error in the 2-norm. }
\label{tab:TTK-table}
\end{table}
It can be observed via Table \ref{tab:TTK-table} that the 2-norm error (\textbf{TT-Error}) obtained by TTK exhibits a close correspondence to the optimal 2-norm error given by SVD (\textbf{SVD-Error}) for the same compression rank (\textbf{rank}) given by TTK. Furthermore, the 2-norm error obtained by TTK is also below or within a factor of ${10}$ of the given tolerance $1 \times 10^{-9}$. The one exception is the SE kernel; this is due to the fact that the conditioning of the cross intersection matrices of the DMRG super cores of the coefficient tensor $\ten{M}$ induced by the SE kernel is sensitive to the distance between the source and target points for a fixed $\ell$. At the same time, TTK is much more efficient (see the column \textbf{TT-time}) compared to the run time of SVD (\textbf{SVD-time}) by a factor of around $98\times$ . It should be noted that TTK is slightly slower compared to another method, such as ACA, which also has a linear computational cost dependence on the number of source and target points, as can be seen from the next set of experiments. However, TTK can be advantageous when the source and target points change in time, but the boxes $\mc{B}_s$ and $\mc{B}_t$ do not change. This scenario is explained further in~\cite[Section 5.2, Experiment 5]{Saibaba2022}. 
\subsection{Parametric Kernels}
For this experiment, we take the dimension $d = 3$ and $n = 32$ Chebyshev nodes in both the spatial and parameter dimensions. Other problem settings are as follows:
\begin{enumerate}
    \item \textbf{Kernels}: We consider  the SE, Thin-plate spline, and Multiquadric kernels with one length scale parameter $\boldsymbol\theta = (\ell)$ and the Mat\'ern kernel with two parameters for length scale $\ell$ and $\nu$  (i.e., $\boldsymbol{\theta} = (\ell, \nu)$).
    \item \textbf{Boxes}: We define $D_b$ to be the distance between the bottom left corners of the boxes $\mathcal{B}_{s}$ and $\mathcal{B}_{t}$.   We fix $\mathcal{B}_{s} = [0,1]^{3}$ and we consider 
$\mathcal{B}_{t} = [1, 2]^{3}$, so that the boxes  $\mathcal{B}_{s}$, $\mathcal{B}_{t}$ are weakly admissible. 
   We take $N_s = N_t = 5000$ points from $\mathcal{B}_{s}$ and $\mathcal{B}_{t}$, uniformly at random, to synthetically construct the source and target points $\mathcal{X}$ and $\mathcal{Y}$.
    \item \textbf{Parameters}: We define the parameter space as follows: for the Squared-Exponential, Thin-plate spline, and Multiquadric kernels. The parameter spaces $\Theta_{1}$ is set to $[\frac{D_b}{2}, D_b]$, while for the Mat\'ern kernel, $\Theta_{2}$ is set to $[\frac{D_b}{2}, D_b] \times [\frac{1}{2}, 3]$. To compute the error, we randomly sample $300$ points uniformly from these parameter spaces.
   
\end{enumerate}
In Table~\ref{tab:PTTK-table} we report the time for the offline computations, offline storage costs, online time, and relative error in the Frobenius norm. We consider the four different parametric kernels described above with tolerances $10^{-4}, 10^{-6}, 10^{-8}$ for the Algorithm~\ref{alg:offline}. We observe that the error in the computed approximations is within a factor of $10$ compared to the requested tolerance. We also note that the storage and the offline computational time increase with the decrease in tolerance (i.e., higher accuracy). Furthermore, the storage of Mat\'ern is higher than that of the other kernels because it has one additional parameter dimension (for the Mat\'ern kernel, $D=8$ but for the other kernels $D=7$). We draw attention to the fact that the cost of the online stage of the PTTK method (Algorithm \ref{alg:online}) does not depend on $N_s$ nor $N_t$; hence, the online time does not show significant changes for an increasing number of source and target points.
\begin{table}[!ht]
\small 
\centering
\begin{tabular}{c|c|c|c|c|c}
\hline
\textbf{Kernel}              & \textbf{Tol}   & \begin{tabular}{@{}l@{}}
                   \textbf{Offline}\\
                   \textbf{Time (s)}\\
                 \end{tabular}     & \begin{tabular}{@{}l@{}}
                   \textbf{Storage}\\
                   \textbf{(MB)}\\
                 \end{tabular}  &  \begin{tabular}{@{}l@{}}
                   \textbf{Online}\\
                   \textbf{Time (s)}\\
                 \end{tabular}  & \textbf{Error}       \\ \hline
Mat\'ern             & \num{1.00e-4} & \num{12.2}    & \num{3.02e0}   & \num{3.20e-3} & \num{1.84e-4} \\ 
Mat\'ern             & \num{1.00e-6} & \num{50.8}   & \num{12.4}   & \num{3.88e-3} & \num{1.86e-6} \\ 
Mat\'ern             & \num{1.00e-8} & \num{337}   & \num{43.7}   & \num{7.06e-3} & \num{1.84e-8} \\ \hline
SE & \num{1.00e-4} & \num{2.70e0}    & \num{3.16e0}   & \num{1.63e-3} & \num{4.42e-4} \\ 
SE & \num{1.00e-6} & \num{4.26e0}   & \num{8.04e0}   & \num{1.78e-3} & \num{4.50e-6} \\ 
SE & \num{1.00e-8} & \num{7.41e0}   & \num{16.7}   & \num{2.27e-3} & \num{4.54e-8} \\ \hline
Multiquadric      & \num{1.00e-4} & \num{2.40e0}    & \num{.997} & \num{1.59e-3} & \num{7.21e-5} \\ 
Multiquadric      & \num{1.00e-6} & \num{2.99e0}   & \num{3.11e0}   & \num{1.63e-3} & \num{7.52e-7} \\ 
Multiquadric      & \num{1.00e-8} & \num{4.55e0}   & \num{7.20e0}   & \num{1.75e-3} & \num{8.80e-9} \\ \hline
Thin-plate-spline  & \num{1.00e-4} & \num{2.54e0}    & \num{1.52e0}   & \num{1.60e-3} & \num{1.59e-3} \\ 
Thin-plate-spline  & \num{1.00e-6} & \num{3.31e0}   & \num{3.81e0}   & \num{1.65e-3} & \num{5.74e-6} \\ 
Thin-plate-spline  & \num{1.00e-8} & \num{6.49e0}   & \num{9.08e0}   & \num{1.82e-3} & \num{6.69e-8} \\ 
\end{tabular}
\caption{ PTTK applied to various kernels with different tolerance values. The source box and target box are weakly admissible. The column \textbf{Online Time} is averaged over $300$ parameter samples, while the column \textbf{Error} reports the maximum relative error in the Frobenius norm over the same $300$ samples.}
\label{tab:PTTK-table}
\end{table}
\begin{table}[!ht]
\small
\centering
\begin{tabular}{c|c|c|c|c}
\hline
\textbf{Kernel} & \textbf{Tol} & \textbf{Online Time (s)} & \textbf{Speed Up} & \textbf{Error} \\ \hline
Mat\'ern             & \num{1.00e-4} & \num{.265}   & \num{82.7e0} & \num{1.39e-3} \\ 
Mat\'ern             & \num{1.00e-6} & \num{.671}   & \num{173} & \num{2.02e-5} \\ 
Mat\'ern             & \num{1.00e-8} & \num{1.43e0}    & \num{202} & \num{1.18e-7} \\ \hline
SE & \num{1.00e-4} & \num{2.42e-2}   & \num{14.9} & \num{7.25e-4} \\ 
SE & \num{1.00e-6} & \num{5.60e-2}   & \num{31.4} & \num{4.18e-6} \\ 
SE & \num{1.00e-8} & \num{.114}   & \num{50.1} & \num{4.32e-8} \\ \hline
Multiquadric     & \num{1.00e-4} & \num{7.91e-3}   & \num{4.96e0} & \num{8.10e-4} \\ 
Multiquadric      & \num{1.00e-6} & \num{2.15e-2}   & \num{13.2e0} & \num{4.84e-6} \\ 
Multiquadric      & \num{1.00e-8} & \num{5.14e-2}   & \num{29.4e0} & \num{4.90e-8} \\ \hline
Thin-plate-spline  & \num{1.00e-4} & \num{2.09e-2}   & \num{13.1e0} & \num{8.51e-4} \\ 
Thin-plate-spline  & \num{1.00e-6} & \num{5.69e-2}   & \num{34.4e0} & \num{4.07e-5} \\ \
Thin-plate-spline  & \num{1.00e-8} & \num{.158}   & \num{86.6e0} & \num{9.81e-8} \\ 
\end{tabular}
\caption{ ACA applied to various kernels with different tolerance values, such that the source box and target box are weakly admissible. The column \textbf{Online Time} is averaged over 300 samples, while the column \textbf{Error} is the maximum relative error in the Frobenius norm over 300 samples. The \textbf{Speedup} column represents the relative speedup compared of PTTK method compared to the ACA method.}
\label{tab:ACA-table}
\end{table}
\paragraph{Comparison with ACA} In Table~\ref{tab:ACA-table}, we compare the online time for PTTK with another method, ACA (Algorithm 1 in \cite{liu2020parallel}), which also has linear complexity in the number of source and target points. The tolerance here refers to the tolerance used for ACA (Frobenius norm) and Algorithm~\ref{alg:offline}.   Table \ref{tab:ACA-table} indicates a close correspondence in the maximum Frobenius norm error (\textbf{Error}) for the ACA methods, typically within a factor of $10$. Therefore, the comparison to PTTK is fair, even though the tolerances have different meanings in both contexts. However, the PTTK method significantly outperforms ACA in terms of online time (\textbf{Online Time}), as illustrated by the \textbf{Speedup} column in Table \ref{tab:ACA-table} with speedups ranging from $4\times$ to $202\times$. In particular, the speedup for the Mat\'ern kernel is quite drastic, since the kernel evaluations for this kernel is relatively expensive. The storage cost of PTTK (\textbf{Storage}) is notably low, typically much less than $45$ MB, as evident from Table \ref{tab:PTTK-table}. These observations emphasize the benefits of our method, which achieves lower online times with minimal storage overhead and minimal offline computation time and, at the same time, maintains accuracy comparable to ACA, which has a similar asymptotic computational complexity. A comparison to another parametric low-rank approximation is given in \cref{sec:tucker}.

\subsection{Global low-rank approximation}
In this set of experiments, we consider the performance of the PTTK-Global method. We take the dimension $d = 3$ and $n=27$ Chebyshev nodes in both the spatial and parameter dimensions. Other problem settings are as follows:
\subsubsection{Synthetic}
We consider the case where the source and target boxes are the same. Other relevant problem settings are described below:
\begin{enumerate}
    \item \textbf{Kernels}: We consider both the Squared-Exponential (SE) and the Multiquadric kernels with one length scale parameter $\vec{\theta} = (\ell)$.
    \item \textbf{Boxes}: We define the source and target boxes to be $\mathcal{B}_{s} = \mc{B}_t= [0, 1]^{3}$. We now define $D_b =\sqrt{3}$ which is the length of the diagonal of $\mathcal{B}_{s}$. We take $N_s = N_t = 10^5$ points from $\mathcal{B}_{s}$, uniformly at random, in order to synthetically construct the source and target points $\mathcal{X}$ and $\mathcal{Y}$.
    \item \textbf{Parameters}: We fix the parameter space to be $\Theta = [.2D_b, D_b]$ for either kernel. 
    \item \textbf{Error}: We report the subsampled relative error in the Frobenius norm averaged over $300$ parameter samples with $|\mathcal{\hat{X}}| = 500$. The set $\mathcal{\hat{X}}$ is constructed by taking the points uniformly at random from $\mathcal{X}$.
\end{enumerate}
The two kernels are chosen because the Squared-Exponential (SE) and Multiquadric kernels are sufficiently smooth, and the SE kernel induces a positive semi-definite matrix while the Multiquadric kernel does not. Our PTTK-Global method can be applied to indefinite kernels by removing line 7 in Algorithm \ref{alg:pttk_global}. We also note that it has been observed in \cite{Cai} that low-rank approximation schemes employing the Nystr\"om method, with k-means and uniform sampling, face difficulties when handling kernels lacking positive definiteness (e.g., Multiquadric) while our methods do not seem to face such difficulties.
\begin{table}[!ht]
\small 
\centering
\begin{tabular}{c|c|c|c|c|c} \hline 
\textbf{Kernel} & \begin{tabular}{@{}l@{}}
                   \textbf{Offline}\\
                   \textbf{Time (s)}\\
                 \end{tabular}  & \textbf{{Error}} & \textbf{{Storage} (MB)} & \begin{tabular}{@{}l@{}}
                   \textbf{Online}\\
                   \textbf{Time (s)}\\
                 \end{tabular} & \textbf{{Rank}} \\ \hline
\multicolumn{6}{c}{\textbf{PTTK-Global-1}} \\ \hline 
 SE & \num{115} & \num{3.72e-6} & \num{485} & \num{5.26e-2} & \num{580} \\ \hline
 Multiquadric & \num{104} & \num{2.55e-6} & \num{375} & \num{3.33e-2} & \num{453} \\
\hline
\multicolumn{6}{c}{\textbf{PTTK-Global-2}} \\ \hline
        SE & \num{115} & \num{8.77e-6} & \num{485} & \num{0.238} & \num{62.5} \\ \hline
        Multiquadric & \num{104} & \num{8.57e-6} & \num{375} & \num{0.166} & \num{51.9} \\
 
\end{tabular}
\caption{PTTK-Global applied to the SE and Multiquadric kernels. The columns \textbf{Error} is the mean subsampled relative error measured in the Frobenius norm over the $300$ samples in sample space $\Theta$, \textbf{Online Time} is the mean online time over the sample space, and  \textbf{Rank} is the mean compression rank over the sample space. }\label{tab:global_synthetic}
\end{table}

For all experiments in Table \ref{tab:global_synthetic}, the tolerance given to Algorithm \ref{alg:greedy} is fixed to be $\epsilon_{\rm tol}=10^{-5}$ during the Offline Stage of Algorithm \ref{alg:offline}. In Table \ref{tab:global_synthetic}, the PTTK-Global-1 and PTTK-Global-2 methods achieve a mean relative error that corresponds closely with our desired tolerance given to the TT-cross approximation within a factor of $10$. With Table \ref{tab:global_synthetic}, the trade-off between compression and online time is clear; PTTK-Global-1 is faster than PTTK-Global-2 but has a higher final rank. The average relative error of PTTK-Global-1 is lower than PTTK-Global-2 due to the additional compression step in PTTK-Global-2 that increases the error.

\subsubsection{Real-world data set}
In this section, we compare our proposed methods, PTTK-Global-1 and PTTK-Global-2, to several alternative low-rank approximation methods designed for symmetric positive semi-definite matrices. These methods include RP-Cholesky, introduced in \cite{RPCholesky}, and Nystr\"om's method with uniform sampling, as implemented in \cite{scikit-learn}. Both methods have linear complexity in the number of source and target points. Additionally,  note that RP-Cholesky and Nystr\"om's method can also be applied to problems in higher spatial dimensions compared to our method.

This experiment is motivated by an application to Gaussian Process (GP); however, in this experiment we are solely focused on the low-rank approximation and not the inference or hyperparameter estimation. We are provided with daily precipitation data from $5500$ US weather stations in $2010$, denoted by dataset $\mat{E}$. The dataset $\mat{E}$ is a matrix of size \(628,474 \times 3\), where each row represents a unique observation in two spatial dimensions and a time dimension (i.e., with $d=3$). We now define the source points to be $\mathcal{X} = \{\vec{x}_1, \dots, \vec{x}_{N_s}\}$ where $\vec{x}_{i} $ is the $i$th row of $\mat{E}$ such that the points in $\mathcal{X}$ are normalized (zero mean and unit variance) and $N_s = 628,474$. Furthermore, we take $\mathcal{X} = \mathcal{Y}$ so the kernel matrix is symmetric positive definite. For the setup, we take the spatial dimension to be $d = 3$ and the parameter dimension to be $d_\theta = 1$; subsequently, $n = 32$ Chebyshev points are taken in both the spatial and parameter dimension. The additional problem settings are outlined as follows: 
\begin{enumerate}
    \item \textbf{Kernel}: We consider the positive definite SE kernel with one length scale parameter $\vec{\theta}  = (\ell)$. We define the parameter space as $\Theta = [.4\rho, \rho]$ where $\rho$ is defined to be the radius of the set $\mathcal{X}$; as in,  $\rho = \max_{x \in {\mathcal{X}}}\| \vec{x}\|_{2}$. 
        \item \textbf{Box}: We take $\mc{B}_s = \mc{B}_t$ to be a bounding box containing the points after normalization.
    \item \textbf{Error}: We compute the subsampled relative error with $|\mathcal{\hat{X}}| = 500$. Again, to compute the error, we randomly sample 300 points uniformly from the parameter space $\Theta$.
\end{enumerate}

\begin{table}[!ht]
    \small 
    \centering
    \begin{tabular}{c|c|c|c|c|c} \hline 
        \textbf{Method} & \begin{tabular}{@{}l@{}}
                   \textbf{Storage}\\
                   \textbf{(MB)}\\
                 \end{tabular}  &   \begin{tabular}{@{}l@{}}
                   \textbf{Offline}\\
                   \textbf{Time (s)}\\
                 \end{tabular}  & \begin{tabular}{@{}l@{}}
                   \textbf{Error}\\
                   \textbf{(Mean)}\\
                 \end{tabular} & \begin{tabular}{@{}l@{}}
                   \textbf{Error}\\
                   \textbf{(Max)}\\
                 \end{tabular} & \begin{tabular}{@{}l@{}}
                   \textbf{Online}\\
                   \textbf{Time (s)}\\
                 \end{tabular} \\
        \hline 
        \begin{tabular}{@{}l@{}}
                   PTTK\\
                   Global-1\\
                 \end{tabular}  & \num{2369} & \num{596} & \num{4.28e-6} & \num{3.18e-5} & \num{3.51e-2} \\ \hline
          \begin{tabular}{@{}l@{}}
                   PTTK\\
                   Global-2\\
                 \end{tabular} & \num{2369} & \num{596} & \num{9.96e-6} & \num{3.37e-5} & \num{1.30} \\ \hline
        RP-Cholesky & $\sim$  & $\sim$ & \num{6.47e-5} & \num{1.31e-4} & \num{22.1} \\ \hline
        Nystr\"om & $\sim$ & $\sim$ & \num{3.63e-4} & \num{2.34e-3} & \num{1.07} \\
    \end{tabular}
    \caption{Comparison of PTTK-Global to various low-rank approximation methods for the SE kernel. The columns \textbf{Error} (mean) and \textbf{Error} (max) report the mean/max subsampled relative error taken over the $300$ sample points in $\Theta$. Note that the columns \textbf{Storage} and \textbf{Offline Time} do not apply to RP-Cholesky and Nystr\"om since they do not require pre-computation. } 
\label{tab:GP}
\end{table}
In all trials listed within Table \ref{tab:GP}, the tolerance given to Algorithm~\ref{alg:greedy} is again fixed to be $\epsilon_{\rm tol} = 10^{-5}$, which is called during the Offline Stage in Algorithm~\ref{alg:offline}. The entries in Table \ref{tab:GP} encompass storage capacity, offline time, the mean and maximum errors computed across 300 samples in $\Theta$ utilizing the Frobenius norm, and online time. In each of the 300 trials, the compression rank denoted as $r_{c}$, provided by PTTK-Global-2, is next utilized by RP-Cholesky and Nystr\"om to derive a rank $r_{c}$ approximation of the parametric kernel matrix. This ensures we are comparing the accuracy of the methods for a fixed storage cost. The mean compression rank, averaged over $300$ samples, obtained from PTTK-Global-2 was $91$ and with PTTK-Global-1 was $468$. For both methods, we see that the mean and max approximation error is within a factor of 10 of the given tolerance. 
\paragraph{Comparison with other methods}
Our PTTK-Global-2 method showcases notable advantages over RP-Cholesky and Nystr\"om's method in terms of accuracy, time, or both, although with increased storage requirements and pre-computation time for the same rank. It is over $20\times$ faster than RP-Cholesky and offers $6\times$ better accuracy according to the \textbf{Error} (mean) column. In comparison to Nystr\"om,  PTTK-Global-2 has the same online time, but it provides at least a factor $36\times$ improvement in accuracy based on the same column. The column designated as \textbf{Error} (max) highlights that the maximum error of PTTK-Global-1, PTTK-Global-2 is within a factor of $4$ relative to the prescribed tolerance ($ 10^{-5}$); in contrast, the other methods are contained within a factor exceeding $10$. Lastly, the table indicates that PTTK-Global-1, is at least $30 \times$ faster than all the other methods and at least $600 \times$  faster than RP-Cholesky, in particular. This comes with the caveat that PTTK-Global-1 has a higher compression rank of $468$ compared to the mean compression rank of $91$ for the other methods.

\section{Conclusions}
In this work, we proposed a low-rank black box approximation scheme of parametric kernel matrices via multivariate polynomial approximation and the tensor train decomposition. Our method can be applied to kernel matrices induced by various kernels and varying spatial configurations of source and target points ($\mc{X} = \mc{Y}$ or $\mc{X} \ne \mc{Y}$), if the kernel function is sufficiently smooth over the chosen parameter space.
After an offline stage, our method outperforms existing methods with an efficient online phase that does not depend on $N_s$ nor $N_t$. Our method works well for three spatial and low parameter dimensions; an investigation of the extension to higher dimensions is left for future work. For future work, we can also consider the TT-matrix form of the coefficient tensor ($\ten{M}$), similar to \cite{könig2023efficient} for the SE kernel.   An alternative direction for future work would be applying our method in the context of hierarchical matrices by using our method as a sub-routine to compute hierarchical parametric low-rank approximations.

\appendix
\section{TT-cross Approximation} \label{sec:tt_cross_approx}
We aim to obtain an approximate tensor train decomposition  of $\ten{X} \in \R^{n_1 \times \dots n_N}$ without explicitly forming $\ten{X}$ first. To this end, we introduce the TT-cross Approximation method, first proposed in \cite{oseledets2010ttcross}, which serves as a multidimensional analog to matrix cross approximation. Our description is heavily influenced by~\cite[Section 2.3]{strossner2022approximation}. 

\paragraph{Cross Approximation} 
We now describe the Cross Approximation method for an arbitrary matrix $\mat{A} \in \R^{m \times n}$. We first suppose that $\mat{A}$ is exactly rank $r_s$, then for the correct choice of the index sets $\mc{I}, \mc{J}$ there exists a skeleton decomposition of $\mat{A}$ where 
$$\mat{A} = \mat{A}(:, \mc{J})[\mat{A}(\mc{I},\mc{J})]^{-1}\mat{A}(\mc{I}, :) , $$
where $\mc{I}, \mc{J}$ are ordered index sets such that the two matrices $\mat{A}(:, \mc{J}) \in \R^{m \times r_s}$ and $ \mat{A}(\mc{I}, :) \in \R^{r_s \times n}$ contain columns and rows of $\mat{A}$ corresponding to the index sets $\mc{J}$ and $\mc{I}$ respectively. Similarly, $\mat{A}(\mc{I},\mc{J})$ is the submatrix of $\mat{A}$. Now, the task is to find an appropriate choice of $\mc{I}$ and $\mc{J}$ when $\mat{A}$ is approximately rank $r_s$, a good choice is index sets $\mc{I}, \mc{J}$ such that $\mat{A}(\mc{I}, \mc{J})$ is the sub-matrix of maximum volume in $\mat{A}$ as stated in \cite{Goreinov1997, Goreinov2001}. However, finding a sub-matrix of maximum volume is NP-hard, as mentioned in \cite{Civril2009}. Thus, in order to compute $\mc{I}$ and $\mc{J}$, we use a heuristic, specifically Algorithm 2 in \cite{DOLGOV2020106869} to incrementally construct the index sets that define the cross approximation.  In our implementation, we modified this slightly and present this as Algorithm~\ref{alg:mat_cross}.

\begin{algorithm}[!ht]
\caption{Update Cross Approximation}\label{alg:mat_cross}
\begin{algorithmic}[1]
\Require{ Matrix $\mat{A} \in \R^{n \times m}$  with access to its entries  and previous index sets $\mc{I}$ and $\mc{J}$}
\State Generate $\max\{n, m\}$ random samples of $\{1, \dots, n\} \times \{1, \dots, m\}$ without replacement, as $S$
\State Set residual matrix $\mat{R} = \mat{A}(:, \mc{I})\mat{A}(\mc{I}, \mc{J})^{-1} \mat{A}(\mc{J}, :) - \mat{A}$
\State Set  $(l, t) \gets \textrm{arg-max}_{(i, j) \in S}|[\mat{R}]_{i, j}|$
\State Draw a random bit $r_b$ 
\If{$r_b = 1$} 
\State $ t \gets \textrm{arg-max}_{1 \le j \le m}|[\mat{R}]_{l, j}|$
\Else 
\State $l \gets \textrm{arg-max}_{1 \le i \le n}|[\mat{R}]_{i, t}|$
\EndIf
\State $\mc{I} = \mc{I} \cup \{l \}, \mc{J} = \mc{J} \cup \{t\}$
\State \Return $\mc{I}, \mc{J}$
\end{algorithmic}
\end{algorithm}
\paragraph{Nested index sets}
We define two index sets $\mc{I}^{\le k}$ and $\mc{I}^{> k}$ that satisfy 
\begin{align}
\mathcal{I}^{\le k} \subset  & \> \{(i_1,i_2,\dots,i_k): 1 \le  i_j \le n_j , 1 \le j \le k\}, \\ \mathcal{I}^{> k} \subset  & \>\{(i_{k+1},\dots,i_{N}): 1 \le i_j \le n_j ,  k \le j \le N  \}     \end{align}
for $1 \le k \le N-1$, and then we define the corresponding border index sets  $\mathcal{I}^{\le 0} = \mathcal{I}^{> N} = \emptyset$. 
As in the two-dimensional (matrix) case, the index sets are also given a consistent ordering, determined by the TT greedy-cross algorithm; hence, we denote the ordered multi-indices of the index sets with $\mc{I}^{\le k}_{s}$ and $\mc{I}^{ > k}_{t}$ where $1 \le s \le |\mc{I}^{\le k}|, 1 \le t \le |\mc{I}^{ > k}|$. We also require the index sets $\mathcal{I}^{\le 1}, \dots,\mathcal{I}^{\le N-1}$ and $\mathcal{I}^{> 1}, \dots, \mathcal{I}^{ > N -1}$ to be nested, which means for $1 \le k \le N-1$, the following holds
\begin{align*}
    (i_1,i_2,\dots,i_k) \in \mathcal{I}^{\le k} \implies (i_1,i_2,\dots,i_{k-1}) \in \mathcal{I}^{\le k-1}, \\ 
    ( i_{k+1}, \dots,i_{N}) \in \mathcal{I}^{ > k} \implies (i_{k+2}, i_{k+3}, \dots, i_{N}) \in \mathcal{I}^{> k +1}.
\end{align*}
With the index sets, we can define the cross intersection matrix $\mat{X}(\mathcal{I}^{\le k}, \mathcal{I}^{> k }) \in \mathbb{R}^{|\mc{I}^{\le k}| \times |\mc{I}^{ > k}|}$
:$$
[\mat{X}(\mc{I}^{\le k}, \mc{I}^{ > k})]_{i, j} = [\mat{X}^{\{k\}}]_{\overline{\mc{I}^{\le k}_{i}}, \overline{\mc{I}^{> k}_{j}}} , \quad  1 \le i, j \le r_k.
$$

\paragraph{Overview of TT greedy-cross} We present an overview of the TT greedy-cross algorithm introduced in \cite{savo}, a specific TT-cross Approximation method. The TT greedy-cross algorithm must be initialized with index sets and in our implementation, it is initialized with index sets of cardinality $2$; for details, see Appendix~\ref{sec:greedyinit}. The TT-ranks of the TT-cross interpolation constructed by the TT greedy-cross algorithm are related to the cardinalities of the index sets with the relationship 
$$r_k = |\mc{I}^{\le k}| = |\mc{I}^{> k} |, \quad 1 \le k \le N-1.$$

We first define the three-dimensional tensor  $\ten{X}(\mc{I}^{\le k-1}, :, \mc{I}^{> k}) \in \R^{r_{k-1} \times n_k \times r_k}$ for $2 \le k \le N-1$ as follows:
$$[\ten{X}(\mc{I}^{\le k-1}, :, \mc{I}^{> k}) ]_{s, l, t} = x_{\mc{I}^{\le k-1}_{s}, l, \mc{I}^{>k}_{t}}, \quad 1 \le s \le r_{k-1}, 1 \le l \le n_k, 1 \le t \le r_k.$$
The edge case tensors $\ten{X}(\mc{I}^{\le 0}, :, \mc{I}^{> 1}) \in \R^{1 \times n_1 \times r_1}, \ten{X}(\mc{I}^{\le N-1}, n_N, \mc{I}^{> N}) \in \R^{r_{N-1} \times n_N \times 1}$ are defined similarly. To begin, we are given a tolerance $\epsilon > 0$, then the TT greedy-cross algorithm executes a left-to-right sweep, and updates the TT-ranks $r_1,\dots,r_{N-1}$ of the TT-cross approximation of $\ten{X}$ at each sequential iteration $k$ by operating on the reshaped DMRG super cores of size $\mat{X}(\mathcal{I}^{\le k-1}, :, :, \mathcal{I}^{ >  k+1}) \in  \mathbb{R}^{r_{k-1} \times n_k \times n_{k+1} \times r_{k+1}}$. In our implementation, for each $1 \le k \le N-1$, the index sets $\mathcal{I}^{\le k }$, $\mathcal{I}^{> k }$ are incremented by applying Algorithm~\ref{alg:mat_cross} to the reshaped super core $[\mat{X}(\mathcal{I}^{\le k-1}, :, :, \mathcal{I}^{ >  k+1})]^{\{2\}}$  via its entries. Once converged, we use the TT-cross interpolation formula \cite[Equation (7)]{savo} in order to obtain the  TT approximation to $\ten{X}$ constructed via cross approximation:
\begin{equation}\label{eq:int_formula}
\begin{split}
&x_{i_1, i_2, \dots, i_N} \approx  \\ 
&\sum_{s_1,t_1 = 1}^{r_1}\dots \sum^{r_{N-1}}_{s_{N-1},t_{N-1} = 1} \prod^{N}_{k = 1}[\ten{X}(\mathcal{I}_{}^{\le k-1}, :, \mathcal{I}^{ > k  }_{})]_{s_{k-1}, i_k, t_{k}}[\mat{X}(\mathcal{I}^{\le k}, \mathcal{I}^{ > k} ) ]^{-1}_{s_k, t_k},
\end{split}
\end{equation}
where $s_0 = 1$ and $t_N = 1$.  Ideally, we would like the relative error of our TT-cross approximation to be less than or equal to the tolerance $\epsilon$; however, we can only heuristically check for convergence of the TT-cross approximation via sampling. We summarize this discussion in Algorithm~\ref{alg:greedy}. Additionally, the error analysis of TT-cross Approximation has been investigated in the Chebyshev norm \cite{savo, osinsky2019tensor} and Frobenius norm \cite{qin2022error}.

\paragraph{Summary of computational costs} Let us once again assume $n_1 = \dots = n_N = n$ and $r = \underset{i}{\max}~r_i$. As mentioned in Appendix~\ref{sec:greedyinit}, the cost of initializing is $\mc{O}(Nn^2)$ tensor entry evaluations. Furthermore, Algorithm~\ref{alg:greedy} requires $\mc{O}(Nnr^3)$ flops and $\mc{O}(Nnr^2)$ tensor entry evaluations; this is because an application of \eqref{alg:mat_cross} on each DMRG super core requires $\mc{O}(nr^2)$ flops and $\mc{O}(nr)$ tensor entry evaluations, there are $N$ such computations in each sweep, and a maximum of $r$ such sweeps.  Note, if the tensor is defined by entries of the parametric kernel function, then we can replace the term tensor entry evaluations with kernel evaluations.

\begin{algorithm}[!ht]
\caption{TT greedy-cross outline}\label{alg:greedy}
\begin{algorithmic}[1]
\Require 
Order-$N$ tensor $\ten{X} \in \mathbb{R}^{n_1 \times \dots \times n_N}$ with access to its entries. A tolerance $\epsilon > 0$.
\Ensure  The TT approximation, $\ten{\hat{X}} = [\ten{C}_1, \ten{C}_2, \dots, \ten{C}_N]$ 
\State Initialize nested index sets using Algorithm~\ref{alg:greedy_init}
\State Initialize TT-ranks $r_1,\dots,r_{N-1} := 2$ and $r_0, r_N = 1$
\While{The relative error of $\ten{X}$ and $\tenh{X}$ in the Chebyshev norm is greater than $\epsilon$ at sample points}
\For{$k = 1 \dots N - 1$}
\State Apply Algorithm \ref{alg:mat_cross} to the matrix $[\mat{X}(\mathcal{I}^{\le k-1},:, :, \mathcal{I}^{ > k+1})]^{\{2\}}$ with index sets $\mc{I}^{\le k}, \mc{I}^{>k}$, and if applicable update index sets $\mc{I}^{\le k}$ and $\mc{I}^{> k}$.
\EndFor
\EndWhile
\State Use~\eqref{eq:int_formula} to obtain the cores $\ten{C}_{1},\dots,\ten{C}_{N}$ using the index sets $\mathcal{I}^{ \le 1}, \dots, \mathcal{I}^{\le  N-1}$ and $\mathcal{I}^{>  1}, \dots, \mathcal{I}^{>   N-1}$.

\end{algorithmic}
\end{algorithm}

\section{TT greedy-cross initialization}\label{sec:greedyinit} We introduce a heuristic for initializing the index sets that will be used in the TT greedy-cross algorithm since searching over all possible index sets $\mc{I}^{\le k}, \mc{I}^{ > k}$ for $1 \le k \le N-1$ of cardinality two in order to minimize $\max_{1 \le k \le N-1} \textrm{cond}_{2}(\mat{X}(\mc{I}^{\le k}, \mc{I}^{ > k}))$, where $\textrm{cond}_{2}(\cdot)$ is the 2-norm condition number of the matrix, would be computationally infeasible. Numerical simulations indicate that initialization with  index sets of cardinality two, rather than one, is needed to apply TT greedy-cross in our method to the squared-exponential kernel in the non-parametric case; otherwise, the TT greedy-cross algorithm quits before a good solution is found. For the parametric case, it is not needed, but does allow the TT greedy-cross algorithm to converge faster. As for the other kernels listed in Table \ref{tab:kern-table}, one can initialize the greedy-cross algorithm with randomized index sets of cardinality 1 for the problem setup described in the numerical experiments section which would take $\mc{O}(N)$ time.

Our heuristic method begins by initializing the index sets $\mc{I}^{ > k}$ for $1 \le  k \le N-1$ with uniformly randomly chosen indices such that $|\mc{I}^{ > k }| = 2$. We start with $k = 1$ and then construct an index set $\mc{I}^{\le 1}$ that minimizes the 2-norm condition number of the cross intersection matrix  $\mat{X}(\mc{I}^{\le 1}, \mc{I}^{> 1})$ such that $\mc{I}^{ \le 1}_{1} \ne \mc{I}^{ \le 1}_{2}$. Now we are given a index set $\mc{I}^{ \le k}$ where $1 \le k \le N-2$, such that $\mc{I}^{ \le k}_{1} = (a_1, a_2, \dots, a_k)$ and $\mc{I}^{ \le k}_{2} = (b_1, b_2, \dots, b_k)$, we then construct the index set $\mc{I}^{\le k + 1}$:
$$\mc{I}^{\le k +1}_{1} = (a_1, a_2, \dots, a_k, i), \quad  \mc{I}^{ \le k +1}_{2} = (b_1, b_2, \dots, b_k, j), $$
such that the condition number of the matrix $\mat{X}(\mc{I}^{\le k + 1}, \mc{I}^{ > k +1})$ is minimized for all possible choices of $1 \le i, j \le n_{k+1}$. After this left-to-right sweep is done for the left index sets $\mc{I}^{\le k }$ for $1 \le k \le N-1$, we then fix the left index sets and analogously construct the right index sets $\mc{I}^{ > k}$ for $1\le k \le N-1$ with a right-to-left sweep. We then repeat the left-to-right sweep and right-to-left sweep until the predetermined stopping condition is reached. In our implementation, we set \text{max\_it}$=10$, and we take the set of  index sets $\{\mc{I}^{\le k}, \mc{I}^{>k} \}_{k=1}^{N-1}$ with the smallest maximum condition number with respect to its cross-intersection matrices from the set of index sets containing every newly generated index sets after every left-to-right sweep and right-to-left sweep. We formalize this discussion in Algorithm \ref{alg:greedy_init}.

In terms of run time, Algorithm \ref{alg:greedy_init} requires $\mc{O}(Nn^2)$ kernel evaluations and $\mc{O}(Nn^2)$ flops. Recall, the value of max\_it is fixed to be $10$. The outermost for loop goes from $1$ to $\textrm{max\_it}$, then the outer loop is composed of two blocks and each block contains two-level nested for loops that go from $1$ to $N$ and from $1$ to $n^2$.

\begin{algorithm}[!ht]
\small 
\begin{algorithmic}[1]
   \caption{TT greedy-cross: index initialization}\label{alg:greedy_init}
  \Require{A function $\phi:\mathbb{R}^{n_1 \times \dots \times n_N} \rightarrow \mathbb{R}$ defined by the entries of tensor $\ten{X}$,  $\phi(i_1, \dots, i_N) = x_{i_1,\dots,i_N}$, where $\ten{X} \in \mathbb{R}^{n_1 \times \dots \times n_N}$ is a $N$ dimensional tensor. User defined \text{max\_it} number of iterations.}
  \Ensure{Index sets $\mc{I}^{ > k}, \mc{I}^{\le k}$ for $1 \le k \le N-1$}
  \State Initialize empty index lists $\mc{I}^{\le k}, \mc{I}^{> k}$ for $1 \le k \le N$ 
  \State Initialize a empty set C $\gets \{ \} $
  \State Generate two random vector of indices $\vec{q}_{1}, \vec{q}_{2} \in R^{ N}$ such that $1 \le [\vec{q}_{1}]_{j},[\vec{q}_{2}]_{j} \le n_{j+1}$ for $1 \le j \le N-1$. 
  \For{$k = 1 ,\dots, N-1$}
  \State $\mathcal{I}^{ > k} \gets \{([\vec{q}_1]_{k+1}, [\vec{q}_1]_{k+2},  \dots, [\vec{q}_{1}]_{N}),([\vec{q}_2]_{k+1}, [\vec{q}_2]_{k + 2},   \dots, [\vec{q}_{2}]_{N})  \}$
  \EndFor
  \For{$\textrm{counter} = 1, \dots ,\textrm{max\_it} $ }
  \For{$k = 1 ,\dots, N-1$}
  \State Denote the elements of $\mc{I}^{ > k}$ as  $(i_{k+1}, \dots, i_{N})$, and $(i_{k+1}', \dots, i_{N}')$, and the elements of $\mc{I}^{ \le k-1}$ as $(i_1, \dots, i_{k-1})$, and $(i_1', \dots, i_{k-1}' )$
  \For{$(j_1, j_2) \in \{1, \dots, n_k\} \times \{1, \dots, n_k\}$}
  \If{$k = 1$ and $j_1 = j_2$} Continue
  \EndIf
  \State Compute matrix \[\mat{M}^{(j_1, j_2)} = \begin{bmatrix}
    \phi(i_1,  \dots, i_{k-1}, j_1, i_{k+1}, \dots, i_{N}) &  \phi(i_1,  \dots, i_{k-1}, j_1, i_{k+1}', \dots, i_{N}')  \\
     \phi(i_1',  \dots, i_{k-1}', j_2, i_{k+1}, \dots, i_{N})  &  \phi(i_1',  \dots, i_{k-1}', j_2, i_{k+1}', \dots, i_{N}')  \\
\end{bmatrix}\]
  \EndFor
  \State $(j_1',j_2') = \arg\min_{1 \le j_1, j_2 \le n_k} \textrm{cond}_2(\mat{M}^{(j_1, j_2)})$, $c_k =\textrm{cond}_2(\mat{M}^{(j_1', j_2')}) $ \Comment{2-norm condition number}
  \State $\mc{I}^{\le k} \gets \{(i_1, i_2, \dots, i_{k-1}, j_1'), (i_1', i_2', \dots, i_{k-1}', j_2')\}$
  \EndFor
    \State C $\gets \textrm{C} \cup (\{\mc{I}^{\le k} \}_{k = 1}^{N-1}$,$\{\mc{I}^{> k} \}_{k = 1}^{N-1}, \max_{1 \le k \le N-1} c_{k})$ 
  \For{$k = N , \dots, 2$}
   \State Denote the elements of $\mc{I}^{ > k}$ as  $(i_{k+1}, \dots, i_{N})$, and $(i_{k+1}', \dots, i_{N}')$, and denote the elements of $\mc{I}^{ \le k-1}$ as $(i_1, \dots, i_{k-1})$, and $(i_1', \dots, i_{k-1}' )$
    \For{$(j_1, j_2) \in \{1, \dots, n_k\} \times \{1, \dots, n_k\}$}
      \If{$k = N$ and $j_1 = j_2$} Continue
  \EndIf
  \State Compute matrix $$\mat{M}^{(j_1, j_2)} = \begin{bmatrix}
    \phi(i_1, \dots, i_{k-1}, j_1, i_{k+1}, \dots, i_{N}) &  \phi(i_1,  \dots, i_{k-1}, j_2, i_{k+1}', \dots, i_{N}')  \\
     \phi(i_1',  \dots, i_{k-1}', j_1, i_{k+1}, \dots, i_{N})  &  \phi(i_1', , \dots, i_{k-1}', j_2, i_{k+1}', \dots, i_{N}')  \\
\end{bmatrix}$$
  \EndFor
  \State $(j_1',j_2') = \arg\min_{1 \le j_1, j_2 \le n_k} \textrm{cond}_2(\mat{M}^{(j_1, j_2)})$, $c_k =\textrm{cond}_2(\mat{M}^{(j_1', j_2')}) $
  \State $\mc{I}^{> k-1} \gets \{(j_1', i_{k+1}, \dots, i_{N} ), (j_2', i_{k+1}', \dots, i_N') \}$
  \EndFor
    \State C $\gets \textrm{C} \cup (\{\mc{I}^{\le k} \}_{k = 1}^{N-1}$,$\{\mc{I}^{> k} \}_{k = 1}^{N-1}, \max_{2 \le k \le N} c_{k})$
  \EndFor  
  \State
  \Return Index sets $\mc{I}^{ \le k}, \mc{I}^{> k}$ for $1 \le k \le N-1$ from the set of tuples C  such that the maximum condition number of the cross-intersection matrices is the smallest (third coordinate of the tuple).
  \end{algorithmic}
\end{algorithm}

\section{Comparison with Tucker}\label{sec:tucker}
This section compares our PTTK method with the method described in Section 4.4 of \cite{Saibaba2022} in spatial dimensions of two and three ($d = 2, 3$). We will refer to this method as parametric Tucker. We take $n = 12$ Chebyshev node for both methods; this amounts to the parametric Tucker method having to construct a tensor of size $12^8$, which requires approximately $3.4$GB of memory to store, assuming eight bits of storage for a floating point number. During the experiment, $49$GB of RAM was required.  The results for this particular numerical experiment have been obtained using the Hazel Cluster at North Carolina State University. 

To further make the comparison meaningful between both methods, we take the compression rank given by the PTTK method, denote it as $r_{c}$ and compute a Tucker approximation of $\ten{M}$ with the same rank $\lceil (r_c + 1)^{1/d} \rceil $ in each mode. Subsequently, the compression rank of the parametric Tucker method ends up being greater than or equal to the compression rank of the PTTK method. Such an approximation is computed using Higher Order Orthogonal Iteration (HOOI) using the TensorLy Python library \cite{tensorly}. This allows us to make both methods' storage and compression strength similar so that meaningful comparisons can be made. Other problem settings are defined as follows:
\begin{enumerate}
    \item \textbf{Kernels:} We consider the Exponential kernel with one length scale parameter $\vec{\theta} = (\ell)$ and the Mat\'ern kernel with two parameters for length scale $\ell $ and $\nu$ (i.e, $\vec{\theta} = (\ell, \nu)$).
    \item \textbf{Boxes:} We define $D_b$ to be the distance between the bottom left corners of the boxes $\mc{B}_{s}$ and $\mc{B}_{t}$. We fix $\mc{B}_{s} = [0, 1]^{d}$ and $\mc{B}_{t} = [2, 3]^{d}$. We take $N_s = N_t = 5000$ points from $\mc{B}_{s}$ and $\mc{B}_{t}$, uniformly at random, to synthetically construct the source and target points $\mc{X}$ and $\mc{Y}$.
     \item \textbf{Parameters:} For the Exponential kernel, the parameter space is set to be $\Theta_1 = [\frac{D_b}{2}, D_b]$ while for the Mat\'ern kernel it is $\Theta_2 = [\frac{D_b}{2}, D_b] \times [\frac{1}{2}, 3]$. To compute the error, we randomly sample $300$ points uniformly from the parameter spaces.
\end{enumerate}
\begin{table}[ht]
\label{tab:PTTK_Tuck_dim3}
\small
\centering
\begin{tabular}{c|c|c|c|c}
\hline
\textbf{Kernel}      & \textbf{Tol}         & \textbf{Offline Time (s)}     & \textbf{Kernel Evals}  & \textbf{Error}         \\ \hline
\multicolumn{5}{c}{\textbf{PTTK}} \\ \hline
Exponential & \num{1E-4}  & \num{5.47E-01}  & \num{23609}    & \num{7.89E-05}  \\ \hline
Exponential & \num{1E-08}   & \num{8.56E-01}  & \num{237857}   & \num{7.85E-09}  \\ \hline
Mat\'ern      & \num{1E-4}  & \num{8.3E-01}  & \num{66254}    & \num{6.93E-05} \\ \hline
Mat\'ern      & \num{1E-08}   & \num{3.11E+00}  & \num{1253435}  & \num{7.71E-07} \\ \hline
\multicolumn{5}{c}{\textbf{ Tucker}} \\ \hline
Exponential & \num{1E-4}  & \num{3.25E+01}  & \num{35831808} & \num{7.77E-05} \\ \hline
Exponential & \num{1E-08}   & \num{3.37E+01}  & \num{35831808} & \num{2.16E-06} \\ \hline
Mat\'ern      & \num{1E-4}  & \num{8.38E+02}  & \num{429981696}& \num{3.1E-04} \\ \hline
Mat\'ern      & \num{1E-08}   & \num{8.54E+02}   & \num{429981696}& \num{1.73E-05} \\ \hline
\end{tabular}
\caption{Comparison of exponential and Mat\'ern kernels using PTTK and Tucker methods with spatial dimension equal to three $(d = 3)$.}
\end{table}

\begin{table}[ht]
\label{tab:PTTK_Tuck_dim2}
\small
\centering
\begin{tabular}{c|c|c|c|c}
\hline
\textbf{Kernel}      & \textbf{Tol}         & \textbf{Offline Time (s)}     & \textbf{Kernel Evals}  & \textbf{Error}         \\ \hline
\multicolumn{5}{c}{\textbf{PTTK}} \\ \hline
Exponential & \num{1E-4}  & \num{3.38E-01}  & \num{7590}     & \num{1.06E-04} \\ \hline
Exponential & \num{1E-08}   & \num{3.92E-01}  & \num{30335}    & \num{6.41E-09} \\ \hline
Mat\'ern      & \num{1E-4}  & \num{4.39E-01}  & \num{20011}    & \num{6.37E-05} \\ \hline
Mat\'ern      & \num{1E-08}   & \num{8.12E-01}  & \num{177036}   & \num{7.56E-07} \\ \hline
\multicolumn{5}{c}{\textbf{ Tucker}} \\ \hline
Exponential & \num{1E-4}  & \num{2.79E-01}   & \num{248832}   & \num{1.23E-04} \\ \hline
Exponential & \num{1E-08}   & \num{2.89E-01}  & \num{248832}   & \num{2.51E-07} \\ \hline
Mat\'ern      & \num{1E-4}  & \num{4.32E+00}  & \num{2985984}  & \num{2.28E-05} \\ \hline
Mat\'ern      & \num{1E-08}   & \num{4.36E+00}   & \num{2985984}  & \num{1.66E-06}  \\ \hline
\end{tabular}
\caption{Comparison of exponential and Mat\'ern kernels using PTTK and Tucker methods with spatial dimension equal to two $(d = 2)$.}
\end{table}

Both methods have a similar storage cost and online time, within a factor of three. Observe that the PTTK method consistently outperforms the Tucker method in the number of kernel evaluations; in particular, for $d = 3$, we observe that the Tucker method uses $103 \times $ to $5960 \times $ as many kernel evaluations when compared to the PTTK method. The PTTK method uses fewer kernel evaluations is to be expected, since in contrast to the Tucker approach, the PTTK method uses TT-cross approximation to approximate the coefficient tensor $\ten{M}$, which doesn't require a full evaluation of the tensor. For $d = 2$, we can see that the offline time between the PTTK and Tucker methods is similar. However, for $d = 3$, we can see that the offline phase of the PTTK method is $34 \times$ to $1109 \times$ faster. In particular, for the Mat\'ern kernel, the offline phase of the PTTK method is faster because the kernel evaluations are relatively expensive compared to the exponential kernel. For both spatial dimensions, the relative error of the PTTK method is at least as good as the Tucker method and, in the best-case scenario, more accurate by a factor of at least $10$.

\bibliographystyle{abbrv}
\bibliography{refs}

\end{document}